\documentclass[a4paper,11pt,oneside]{article}
\usepackage{amsmath,amssymb,amsthm}
\usepackage[matrix,graph]{xy}

\newtheorem{thm}{Theorem}[section]
\newtheorem{mainthm}{Main Theorem}

\newtheorem{cor}[thm]{Corollary}

\newtheorem{defn}[thm]{Definition}
\newtheorem{lem}[thm]{Lemma}

\newtheorem{rem}[thm]{Remark}

\def\Coh{\mathop{\mathrm{Coh}}\nolimits}
\def\Spec{\mathop{\mathrm{Spec}}\nolimits}
\def\End{\mathop{\mathrm{End}}\nolimits}

\def\Q{\mathbb{Q}}

\def\C{\mathbb{C}}
\def\R{\mathbb{R}}

\def\e{\varepsilon}

\def\tor{\mathop{\mathrm{tor}}\nolimits}

\def\Hom{\mathop{\mathrm{Hom}}\nolimits}
\def\diag{\mathop{\mathrm{diag}}\nolimits}
\def\Supp{\mathop{\mathrm{Supp}}\nolimits}

\newcommand{\mo}{\mathcal{O}}

\newcommand{\PP}{\mathbb{P}}

\newcommand{\Z}{\mathbb{Z}}
\newcommand{\Pic}{\operatorname{Pic}}

\def\cl{\mathop{\mathrm{cl}}\nolimits}

\def\Ext{\mathop{\mathrm{Ext}}\nolimits}

\def\im{\mathop{\mathrm{im}}\nolimits}

\def\coker{\mathop{\mathrm{coker}}\nolimits}
\def\dim{\mathop{\mathrm{dim}}\nolimits}

\def\deg{\mathop{\mathrm{deg}}\nolimits}
\def\rank{\mathop{\mathrm{rank}}\nolimits}
\def\rk{\mathop{\mathrm{rk}}\nolimits}

\def\id{\mathop{\mathrm{id}}\nolimits}

\def\t{\times}
\def\X{\mathcal{X}}

\def\mbi#1{\boldsymbol{#1}}
\newcommand{\mk}{\mathfrak}
\newcommand{\mc}{\mathcal}
\newcommand{\mb}{\mathbb}

\newcommand{\gr}{\text{gr}}

\def\vol{\mathop{\mathrm{vol}}\nolimits}
\newcommand{\hooklongrightarrow}{\lhook\joinrel\longrightarrow}
\def\Cone{\mathop{\mathrm{Cone}}\nolimits}
\def\DG{\mathop{\mathrm{DG}}\nolimits}

\newcommand{\rig}{\operatorname{rig}}

\begin{document}
\title{ Frobenius morphisms and derived categories on two dimensional toric Deligne--Mumford stacks}
\author{Ryo Ohkawa and Hokuto Uehara}
\date{}


\maketitle
\begin{center}
\emph{Dedicated to Professor Yujiro Kawamata on the occasion of his 60th birthday}
\end{center}
\begin{abstract}
For a toric Deligne--Mumford (DM) stack $\mc X$, 
we can consider a certain generalization of the Frobenius endomorphism.
For such an endomorphism $F\colon \mc X\to \mc X$ 
on a $2$-dimensional toric DM stack $\mc X$,
we show that the push-forward $F_{\ast}\mc O_{\mc X}$ of 
the structure sheaf  generates the bounded derived category 
of coherent sheaves on $\mc X$.

We also choose a full strong exceptional 
collection from the set
of direct summands of $F_{\ast}\mc O_{\mc X}$ 
in several examples of two dimensional toric DM orbifolds 
$\mc X$.
\end{abstract}


\section{Introduction}\label{sec:1}

There are many smooth projective varieties $X$ 
(or more generally algebraic stacks)
whose derived categories $D^b(X)$
have elements generating $D^b(X)$ (see Definition \ref{def:generator}). 
The direct sums of full exceptional collections give
examples of such. Furthermore for a given  
vector bundle $\mc E$, if $\mc E$ generates $D^b(X)$ and it satisfies that 
$\Hom _{D^b(X)}^{i}(\mc E,\mc E)=0$ for $i\ne 0$,  
the derived category $D^b(X)$ is equivalent to the derived category 
of the module category of its endomorphism algebra 
$\End_{D^b(X)}(\mc E)$ (cf. \cite[Lemma~3.3]{TU}), 
and then we can apply the representation theory of
finite dimensional algebras to study $D^b(X)$. It is important  
to find a ``good'' generator of a given triangulated category.

In \cite{Bo}, Bondal announces that there exists a derived 
equivalence between abelian categories 
of coherent sheaves on a toric variety $X$ and constructible sheaves on 
a real torus with a stratification associated to $X$. 
He uses the assertion (without proof)  that the Frobenius 
push-forward $F_*{\mc O_{X}}$
of the structure sheaf  generates 
the bounded derived category $D^b(X)$.
The purpose of this note is to give a rigorous proof of it 
for the $2$-dimensional stacky case:


\begin{mainthm}[=Theorem \ref{thm}]\label{mainthm}
Let $\X$ be a  two dimensional toric Deligne--Mumford (DM) stack. 
Then the vector bundle $F_*{\mc O_{\X}}$ generates $D^b(\X)$ 
for a Frobenius morphism $F$ on $\X$ with a sufficiently divisible degree.
\end{mainthm} 
The generators we find in Theorem \ref{thm} 
get along well with the birational geometry 
in some sense (cf. Lemma \ref{lem}).
Actually we make full use of the birational 
geometry to obtain Theorem \ref{thm}.
The proof of Theorem \ref{thm} is divided into $4$ steps:

\begin{enumerate}
\item
First we reduce the proof  to the case $\X$ is a toric DM orbifold. 
This step works for arbitrary dimensional case.


\item
We introduce the notion of the 
\emph{associated weighted blow up} 
on a toric DM orbifold $\X$ with a weighted blow up 
on a toric variety $X$ which is 
the coarse moduli space of $\X$.
In the two dimensional case we take a toric resolution 
of $X$, and consider the associated birational morphism on $\X$.
We use Lemma \ref{lem} for this morphism 
to reduce the proof to the case $X$ is smooth.

\item
Use the strong factorization theorem 
to connect $X$ and $\PP^2$ by birational morphisms, 
and consider the associated birational morphisms.
We use Lemma \ref{lem} and a technical Lemma 
\ref{keylemma} to reduce the proof to the case $\X$
is a root stack of $\PP^2$. 

\item
Finally we show the statement for the case 
$\X$ is a root stack of $\PP^2$.
\end{enumerate}

The construction of this note is as follows.
In \S\ref{sec:2} we recall the construction and some fundamental 
properties of toric DM stacks, following \cite{BCS}.
In \S\ref{sec:4} we recall the notion of a \emph{root stack} 
and \emph{rigidification} of toric DM stacks.
In \S\ref{sec:3} we explain how to
compute the direct summands of the 
Frobenius push-forward $F_*{\mc O_{\X}}$ 
for toric DM stacks $\X$, 
and put the step (i) into practice. 
In \S\ref{sec:root-plane}  we put the step (iv) into practice.
In \S\ref{sec:5} we define the 
\emph{associated birational morphisms} as in step (ii), and 
put the step (ii) into practice.
In \S\ref{sec:mainthm} we put the step (iii) into practice
and complete the proof of Theorem \ref{thm}.
In \S\ref{sec:6} we use Theorem \ref{thm} 
to show existences of full strong exceptional collections 
in several examples.

We freely use terminology defined in \S2.1 and \S2.2 after these subsections.
We always work over the complex number field $\C$, and note that 
(certain generalized) Frobenius morphisms can be defined on toric varieties 
over $\C$ (or actually any fields, see \S\ref{subsec:Frobenius}).

The authors began this research during their stay 
in the Max Planck Institute for Mathematics.
They appreciate the Max Planck Institute for Mathematics 
for their hospitality and stimulating environment. 
R.O. thanks Zheng Hua, Akira Ishii, Hiroshi Iritani, 
Yoshiyuki Kimura, Hiraku Nakajima, So Okada and Kazushi Ueda 
for useful discussions.
He also thanks Research Institute for Mathematical Sciences, 
Kyoto University, and Institut des Hautes \'Etudse Scientifiques.
He thanks Isamu Iwanari for answering his question about 
Remark~\ref{rem:stackyfan}.
He is supported by GCOE `Fostering top leaders in mathematics',
Kyoto University.
H.U. is supported by the Grants-in-Aid 
for Scientific Research (No.23340011). 
The authors are very grateful to the referees for the careful reading of the paper.

\paragraph{Notation} 
For any vector space $\C^s$ and any $m\in \Z_{>0}$, 
we denote by $\wedge m\colon\C^s\to\C^s$ a map defined 
by $\wedge m(z_1,\ldots,z_s)=(z_1^{m},\ldots,z_s^{m})$.
We denote by $\Z_m$ the cyclic group $\Z/m\Z$.
For any real vector $\mbi v={}^t(v_1,\ldots,v_s)\in\R^s$, we put 
$\lfloor\mbi v\rfloor={}^t(\lfloor v_1\rfloor,\ldots,
\lfloor v_s\rfloor)\in\Z^s$, 
where $\lfloor v_i \rfloor$ is the integer
satisfying $\lfloor v_i\rfloor\le v_i <\lfloor v_i\rfloor +1$.

For a commutative ring $R$ and $l,m\in \Z_{>0}$, 
we denote by $M(l,m,R)$ the set of $l\times m$ matrices  in $R$.
We often identify a matrix $A=(a_{i,j})\in M(l,m,R)$ with a linear map 
$$
R^m\to R^l\quad {}^t(x_1,\ldots, x_m ) \mapsto A{}^t(x_1,\ldots, x_m ).
$$
The symbol $\delta _{i,j}$ stands for the Kronecker delta.
For $a_1,\ldots,a_l\in R$, we denote by $\diag(a_1,\ldots,a_l)$ 
the diagonal matrix $(a_i\delta_{i,j})\in M(l,l,R)$.

For a group homomorphism $\pi \colon N\to N'$ 
between finitely generated abelian groups, 
we denote by $\pi_{\R}\colon N_{\R}\to N'_{\R}$ the associated 
$\R$-linear map between vector spaces $N_{\R}=N\otimes \R$
and $N'_{\R}=N'\otimes \R$. 
For a fan $\Delta$ in $N\otimes \Q$,  
$\Delta(1)$ denotes the set of $1$-dimensional cones in $\Delta$.
$X_\Delta$ denote the toric variety associated with the fan $\Delta$.
  
For a DM stack $\mc X$, we denote by $D^b(\X)$ the bounded 
derived category of coherent sheaves on $\X$.


\section{Toric DM stacks}\label{sec:2}
Toric DM stacks are introduced in \cite{BCS}. 
They are motivated by the Cox's 
construction of toric varieties (\cite{Co}).
In this section we recall some definitions and facts.
See \cite{FMN} and \cite{Iw1} for other definitions.


\subsection{Definitions}\label{sec:2-1}
Let $N$ be a finitely generated abelian group of rank $n$.
We have an exact sequence of abelian groups
$$
0\to N_{\tor}\to N\to \Bar{N}\to 0,
$$
where $N_{\tor}$ is the subgroup of torsion elements in $N$.
A \emph{stacky fan} $\Sigma=(\Delta,\beta)$ 
in $N$ consists of a 
simplicial fan $\Delta$ in $N\otimes \Q$ and a 
group homomorphism  
$$
\beta\colon\Z^s\to N
$$ 
such that for the canonical basis $\mbi f_1,\ldots,\mbi f_s$ of $\Z^s$, 
each $\beta(\mbi f_i)$  generates the cone $\rho_i$ in $N_\R$, where 
$\Delta(1)=\lbrace \rho_1,\ldots,\rho_s\rbrace$.
In this note we always assume that 
$\Delta$ is complete\footnote{Completeness is not 
essential in many arguments below, 
but for simplicity, we assume it.}.

We define a toric DM stack associated to $\Sigma$ as follows.
We take the mapping cone $\Cone(\beta)$ of $\beta$
in the derived category of $\Z$-modules and its derived dual
$\Cone(\beta)^{\star}=\mathbf R\Hom_{\Z}(\Cone(\beta),\Z)$.
We put 
$$
\DG(\beta):=H^1(\Cone(\beta)^{\star})\ \mbox{ and } 
G(=G_{\Sigma}):=\Hom_{\Z}(\DG(\beta),\C^*).
$$ 
%


\begin{rem}
Note that in the Cox's construction \cite{Co}, 
$N$ is torsion free and every $\beta(\mbi f_i)$ is primitive. 
In this case $\DG(\beta)$ is just the 
Chow group $A_{n-1}(X_\Delta)$ of the toric variety $X_\Delta$.
\end{rem}

There exists an exact triangle 
$$
\Cone(\beta)^{\star}\to \mathbf R\Hom_{\Z}(N,\Z)\to \mathbf R\Hom_{\Z}(\Z^s,\Z)(\cong \Z^s),
$$ 
which induces
\begin{equation}\label{cl}
\cl\colon\Z^s \to \DG(\beta).
\end{equation}
Applying $\Hom_\Z( - ,\C^*)$ to it, 
we obtain a map $\Hom_\Z( \cl ,\C^*) \colon G\to (\C^*)^s$.
Hence by the natural action of $(\C^\ast)^s$ on $\C^s$, 
we have a $G$-action on $\C^s$.  

Let $S:=\C[z_1,\ldots,z_s]$ be the coordinate ring of $\C^s$.
For each cone $\sigma$ in $\Delta$, we put 
$z_\sigma=\prod_{\beta_\R(\mbi f_i)\notin \sigma} z_i$ 
and $U_{\sigma}:=\C^s\setminus \lbrace z_{\sigma}=0\rbrace$.
We have a $G$-invariant subspace 
$$
U(=U_{\Sigma}):=\bigcup_{\sigma\in\Delta}U_{\sigma}.
$$
The \emph{toric Deligne--Mumford (DM) stack $\X=\X_{\Sigma}$ 
associated to the stacky fan $\Sigma$} is defined by
the quotient stack 
$$
\X:=[U/G].
$$
This is actually a DM stack by \cite[Proposition 3.2]{BCS}.

When $N$ is torsion free, the stack $\X$ has 
the trivial generic stabilizer group, 
and  we call $\X$ a \emph{toric DM orbifold}. 
In this case, 
let us denote  by $\mbi v_i$ the 
primitive vectors  in $\rho_i\in \Delta(1)$, and also 
denote by $D_i$ (resp. $\mathcal{D}_i$) 
the toric divisors on $X=X_{\Delta}$ (resp. $\X=\X_{\Sigma}$) 
corresponding to the cone $\rho_i$ 
(see their precise definitions in \S \ref{sec:2-2} and \S \ref{subsec:closed}).
Take the positive integers $b_i$ satisfying $\beta(\mbi f_i)=b_i\mbi v_i$.
We often denote the toric DM orbifold $\X$ by 
$$
\X\bigl(X,\sum b_iD_i\bigr).
$$
%


\begin{rem}\label{rem:cox}
If $N$ is torsion free, 
then the map $\Hom_\Z( \cl ,\C^*)$ becomes an inclusion. 
Assume furthermore that every $\beta(\mbi f_i)$ is primitive
and the fan $\Delta$ is non-singular,
i.e. every cone is generated by a subset of a basis of $N$.
Then we can see that
$G$ acts on $U$ freely, and actually the converse is also true. 
In this case, $\X$ is an algebraic space. Since the toric variety 
$X=X_\Delta$ is a coarse moduli space of $\X$ (\cite[Proposition 3.7]{BCS}),
$\X$ coincides with  $X$. 
Thus in this situation, the construction of 
toric DM stacks is same as 
the original construction of toric varieties given in \cite{Co}.

Note that the orbifold $\X\bigl(X,\sum D_i\bigr)$ 
coincides with its coarse moduli  
space $X$ if and only if $X$ is smooth.
In general the orbifold $\X\bigl(X,\sum D_i\bigr)$ 
is the canonical stack with the coarse moduli space $X$.
We put $\mc X^{\text{can}}=\X\bigl(X,\sum D_i\bigr)$. 
\end{rem}

We denote by $\Coh_{G} U$ the category of $G$-equivariant 
coherent sheaves on $U$.  
Then we have an equivalence of categories
\begin{equation}\label{eq}
\Coh\X\cong\Coh_{G} U\quad  \mc F\mapsto \mc F_{U}
\end{equation}
by \cite[(7.21)]{Vi}.
By this equivalence we identify $\Coh\X$ and $\Coh_{G} U$, and define 
$$
H^0(U,\mc F):=H^0(U,\mc F_{U})
$$
for any $\mc F\in\Coh\mc X$.

Define $G^{\vee}$ to be the group   of characters  on $G$, that is, 
$G^{\vee}:=\Hom_{\Z}(G,\C^{\ast})$, and then
note that the map $\cl$ in (\ref{cl}) can be regarded 
as the map from $\Z^s$ to $G^{\vee}$, since
$G^{\vee}$ 
is naturally isomorphic to $\DG (\beta)$. 
The $G$-action on $\C^s=\Spec S$ induces an eigenspace decomposition
\begin{equation}\label{grad}
S=\bigoplus_{\chi\in G^{\vee}}S_{\chi}, 
\text{ where we put }  S_{\chi}:=\bigoplus_{\substack{\mbi k\in \Z^s_{\ge 0}\\ \cl(\mbi k)=\chi}}\C z^{\mbi k}
\end{equation}
and $z^{\mbi k}:=z_1^{k_1}\cdots z_s^{k_s}$ for  $\mbi k={}^t(k_1,\cdots ,k_s)\in \Z^s$.
We call the $G^{\vee}$-graded ring $S$ the \emph{homogeneous coordinate ring} of $\X$. The point in $\X$ corresponding to the $G$-orbit of $(z_1,\ldots,z_s)\in U$ is denoted 
by $[z_1:\cdots:z_{s}]$. We call it the \emph{homogeneous coordinate} of $\X$.

We denote by $\gr S$ the category of $G^\vee$-graded 
finitely generated $S$-modules and by $\tor S$ 
the full subcategory of $\gr S$ consisting of 
$S$-modules annihilated by some powers of the ideal 
$
(z_{\sigma}\mid \sigma\in\Delta).
$
We associate a coherent sheaf $\mc F$ on $\mc X$ with a $G^{\vee}$-graded $S$-module 
$$
M:=\bigoplus_{\chi\in G^{\vee}} H^0(U,\mc F)_{\chi},
$$
where $H^0(U,\mc F)_{\chi}$ is the set of sections having an eigenvalue 
$\chi\in G^{\vee}$ for the action of $G$.
The module $M$ is finitely generated by \cite[Lemma 4.7]{BH1}
and a unique $G^{\vee}$-graded $S$-module, up to $\tor S$, satisfying 
$\Tilde{M}|_{U}\cong \mc F_{U}$ in $\Coh_{G}U$.
Hence we have an equivalence of categories
\begin{equation}\label{grS}
\Coh\X\cong \gr S/\tor S\quad  
\mc F\mapsto M=\bigoplus_{\chi\in G^{\vee}} H^0(U,\mc F)_{\chi}.
\end{equation}
Here $\gr S/\tor S$ is the quotient category 
(cf. \cite[Appendix A.2]{HMS}).


\subsection{Picard groups and toric divisors of toric DM stacks}\label{sec:2-2}
For $\chi\in {G}^{\vee}=\Hom_{\Z}(G,\C^{\ast})$, we define a $G$-action
on the trivial bundle $U\times \C$ by
\begin{align}\label{action1}
 G\times U\times \C\to U\times \C \quad  (g,z,t)\mapsto  (gz,\chi(g)t).
\end{align}
We obtain a ${G}$-equivariant trivial line bundle on $U$,
which defines a line bundle $\mc L_{\chi}$ on $\X$ via the equivalence
(\ref{eq}).
This gives an isomorphism
$
G^{\vee}\cong\Pic \mc X
$
as \cite[Proposition~3.3]{BHu}.
Note that Borisov and Hua show this fact 
under the assumption $N_{tor}=0$,
but a similar proof works in the case $N_{tor}\ne 0$.
Henceforth we often identify $\Pic\X$, $G^\vee$ and $\DG(\beta)$,
and their elements corresponding to each other are denoted as 
\begin{equation}\label{eqn:pic}
\DG(\beta)\cong G^\vee\cong \Pic\X \qquad \Bar{\mbi w}
 \mapsto \chi_{\overline {\mbi w}}=\chi \mapsto \mc L_{\chi}.
\end{equation}

We describe $\Pic \X$ more explicitly.
Take a projective resolution of $N$;
$$
0\to\Z^r\stackrel{A}{\to}\Z^{n+r}\to N\to 0
$$ 
for a matrix $A=(a_{i,j})\in M(n+r,r,\Z)$.
Then the map $\beta\colon \Z^s\to N$ is determined by some matrix $B=(b_{i,j})\in M(n+r,s,\Z)$,
and the mapping cone of $\beta$ is described as  
$$
\Cone(\beta)=\lbrace \cdots \to 0\to \Z^s\oplus\Z^r\stackrel{(B| A)}{\to}\Z^{n+r}\to 0\to \cdots\rbrace,
$$
where the term $\Z^{n+r}$ fits into the degree $0$ part.
Let us denote by $\mbi f_1,\ldots,\mbi f_s$ and $\mbi g_1,\ldots,\mbi g_r$ 
the canonical bases of $\Z^s$ and $\Z^r$ respectively.
Since 
\begin{equation}\label{coker}
\DG(\beta)=H^1(\Cone(\beta)^{\star})= \coker{}^t(B| A),
\end{equation}
we have 
\begin{equation}\label{pic}
\Pic\X\cong \DG(\beta)=
\frac{\Z\mbi f_1^\ast\oplus\cdots\oplus\Z\mbi f_s^\ast\oplus\Z\mbi g_1^\ast\oplus\cdots\oplus\Z\mbi g^\ast_r}
{\left\langle \sum_{j=1}^{s}b_{i,j}\mbi{f^\ast_j}+\sum_{j=1}^ra_{i,j}\mbi{g^{\ast}_j}\mid i=1,\ldots,n+r\right\rangle}.
\end{equation}
We denote by $\Bar{\mbi f^\ast_i}$ and $\Bar{\mbi g^\ast_j}$ the classes of dual bases 
$\mbi f^\ast_i$ and $\mbi g^\ast_j$ in $\DG(\beta)$ respectively.

By (\ref{coker}), the group $G=\Hom_{\Z}(\DG(\beta),\C^{\ast})$ is 
described as 
\begin{equation}\label{G}
G=\left\lbrace\mbi t=\begin{pmatrix}t_1\\\vdots\\t_{s+r}\end{pmatrix}\in(\C^{\ast})^{s+r}\ 
\Bigr|\ 
\prod_{j=1}^s t_j^{b_{i,j}}\times\prod_{j=1}^r t_{s+j}^{a_{i,j}}=1 \text{ for }i=1,\ldots,n+r\right\rbrace.
\end{equation}
The $G$-action on $U\subset \C^s$ is defined by 
\begin{equation}\label{action2}
\mbi t\cdot (z_1,\ldots,z_s)= (t_1z_1,\ldots,t_sz_s)
\end{equation}
for $\mbi t \in G$.
Hence we can see that the group
$$
H(=H_\Sigma):=\{\mbi t\in G\mid t_1=\cdots=t_s=1\}\cong 
\left\lbrace\begin{pmatrix}t_{s+1}\\\vdots\\t_{s+r}\end{pmatrix}\in(\C^{\ast})^{r}\ 
\Bigr|\ 
\prod_{j=1}^r t_{s+j}^{a_{i,j}}=1 \text{ for }i=1,\ldots,n+r\right\rbrace
$$ 
acts on $U$ trivially, and 
it is the generic stabilizer group of  $\X=[U/G]$.
Note that $H$ is a finite group, since $\rk A=r$.

For each $i=1,\ldots, s$,  we have a Cartier divisor
 $\mc{D}_i(=\mc D_i^\X):=[\lbrace z_i=0\rbrace/G]$ (in the notation of
 \S \ref{subsec:closed},
 $\mc{D}_i=[z_i=0]$) on $\X=[U/G]$ corresponding to the ray $\rho_i$, which satisfies that $\mc O_{\X}(\mc{D}_i)\cong \mc L_{\chi_{\Bar{\mbi f^\ast_i}}}$ 
in the notation in (\ref{eqn:pic}). 
We call $\mc{D}_i$ the \emph{toric divisors corresponding to the cone $\rho_i$}.
When $\mc X$ is a variety, we often denote $\mc D_i$ by $D_i$.
Put 
$$
\mbi D(=\mbi D^\X):=\begin{pmatrix}\mc D_1\\\vdots\\ \mc D_s\end{pmatrix}
\text{ and } 
\mbi k\mbi D=\sum_{i=1}^s k_i\mc{D}_i
$$
for $\mbi k=\sum _{i=1}^sk_i\mbi f^\ast_i\in \Z^s$.
We consider a $G^\vee$-graded $S$-module $Sz^{-\mbi k}$ 
(recall that  $z^{-\mbi k}=z_1^{-k_1}\cdots z_s^{-k_s}$ as in (\ref{grad})) corresponding to 
$\mo_\X(\mbi k\mbi D)$ via the isomorphism in (\ref{grS}), where the $G^\vee$-grading of
 $Sz^{-\mbi k}$ is given by $\cl\colon \Z^s\to G^{\vee}$ as in (\ref{grad}).

For $\mbi w\in \Z^{s}\oplus\Z^{r}$ and a 
$G^{\vee}$-graded $S$-module $M$, 
we define a $G^{\vee}$-graded $S$-module 
$$
M(\mbi w)=\bigoplus_{\chi \in G^\vee}M(\mbi w)_{\chi}
$$
by $M(\mbi w)_{\chi}:=M_{\chi+\chi_{\Bar{\mbi w}}}$, 
where $\Bar{\mbi w}$ is the image of $\mbi w$ in $\DG(\beta)$, and 
$\chi_{\Bar{\mbi w}}\in G^{\vee}$ is defined in (\ref{eqn:pic}).
Then we have an isomorphism of $G^{\vee}$-graded $S$-modules 
$
Sz^{-\mbi k}\cong S\left(\mbi k\right).
$

For 
$\mbi l\in \Z^r$ 
denote by 
$$
\mo_{\X}(\mbi k\mbi D)_{\mbi l}
$$
the line bundle which corresponds to the graded $S$-module 
$
S\left(\mbi k+\mbi l\right)
$ 
by (\ref{grS}). 
Here we regard $\mbi k+\mbi l$ as an element of $\Z^s\oplus \Z^r$.
Note that $\mo_{\X}(\mbi k\mbi D)_{\mbi l}$ is the line bundle $\mc L_{\chi_{\overline{\mbi k+\mbi l}}}$ in (\ref{eqn:pic}).
When $\mbi k=\mbi 0$, we put $\mc O_{\mc X,\mbi l}=
\mo_{\X}(\mbi 0\mbi D)_{\mbi l}$.

Henceforth we freely use terminology defined 
in \S\ref{sec:2-1} and \ref{sec:2-2}.
We often use the superscript $'$ for objects associated with a toric DM stack $\X'$. 
For instance, $\Sigma'=(\Delta',\beta')$ stands for the stacky fan in finitely generated abelian group $N'$ defining $\X'$.


\subsection{Closed substacks of toric DM stacks}\label{subsec:closed}

Let $\Sigma=(\Delta,\beta)$ be a stacky fan and $\mc X=\mc X_{\Sigma}$ the associated toric DM stack.
For any non-zero cone $\tau\in \Delta$, we consider the abelian group 
$$
N(\tau):=\frac{N}{\left\langle \beta(\mbi f_i)\mid \rho_i\subset \tau\right\rangle}
$$
and denote by $\pi_{\tau}\colon N\to N(\tau)$ the quotient map. 
We consider the fan 
$$
\Delta_{\tau}:
=\bigl \{ \pi_{\tau, \R}(\sigma) \mid \sigma+\tau\in\Delta, \sigma \in \Delta  \bigr\}
=\bigl \{ \pi_{\tau,\R} (\sigma) \bigm| \tau\subset \sigma \in \Delta \bigr\}
$$
in $N(\tau)_\R$. Then  the map $\pi_{\tau,\R}$ gives  a one to one correspondence between the set
$$
\bigl \{ \rho_i \bigm| \rho_i+\tau\in \Delta, \rho_i\cap \tau=\mbi 0 \bigr\}
$$
 and the set
$$
\Delta_{\tau}(1)=\bigl \{ \pi_{\tau,\R} (\rho_i) \bigm| \rho_i+\tau\in \Delta, \rho_i\cap \tau=\mbi 0 \bigr\}.
$$
By this identification, we regard the set $\Delta_{\tau}(1)$ as a subset of $\Delta(1)$, and hence we have
$
\Z^{\Delta_{\tau}(1)}\subset \Z^{\Delta(1)}=\Z^{s}
$
under the identifications $\Z^{\Delta_{\tau}(1)}\cong \Hom_{\Z}(\Z^{\Delta_{\tau}(1)},\Z)$ and $\Z^{\Delta(1)}\cong \Hom_{\Z}(\Z^{\Delta(1)},\Z)$.
Then we define a stacky fan $\Sigma_{\tau}=(\Delta_{\tau},\beta_{\tau})$ in $N(\tau)$, where
$$
\beta_{\tau}:=\pi_{\tau}\circ \left(\beta |_{\Z^{\Delta_{\tau}(1)}}\right).
$$

The toric DM stack $\X_{\Sigma_{\tau}}$ associated to the stacky fan $\Sigma_{\tau}$ defines a closed substack of $\X$ as follows.

Reorder the set $\Delta(1)$ so that
$\Delta_{\tau}(1)=\lbrace \pi_{\tau, \R}(\rho_{1}),\ldots, \pi_{\tau, \R}({\rho}_{s_\tau}) \rbrace$
and $\tau=\rho_{s_\tau+1}+\cdots +\rho_l$  for some $s_\tau,l$ with $l>s_\tau \ge 0$.
For the homogeneous coordinate $[z_1:\cdots:z_s]$ of $\mc X$, 
we define a subset
$V_{\tau} $ of $U$ and a subgroup $G_{\tau}$ of $G$ by
$$
V_{\tau}:=\lbrace{}^t(z_1,\ldots,z_s)\in U\mid z_{s_\tau+1}=\cdots=z_l=0, z_{l+1}=\cdots =z_{s}=1\rbrace
$$
$$
G_{\tau}:=\lbrace{}^t(t_1,\ldots,t_{s+r})\in G\subset (\C ^*)^{s+r} \mid t_{l+1}=\cdots =t_{s}=1\rbrace.
$$
For a cone $\sigma\in \Delta$ satisfying $\tau \not\subset \sigma$, we have $U_{\sigma}\cap V_{\tau}=\emptyset$, which implies that
$$
V_{\tau}=(\bigcup_{\tau\subset\sigma\in\Delta}U_{\sigma})\cap \{z_{s_\tau+1}=\cdots=z_l=0, z_{l+1}=\cdots =z_{s}=1 \}.
$$
Hence we have a natural isomorphism $U_{\Sigma_{\tau}} \cong V_{\tau}$ induced by the inclusion
$$
\C^{\Delta_{\tau}(1)}=\C^{s_\tau}\hookrightarrow \C^{\Delta(1)}=\C^{s}\quad (z_{1},\ldots,z_{s_\tau})\mapsto 
(z_{1},\ldots,z_{s_\tau},\overbrace{0,\ldots,0}^{l-s_\tau},\overbrace{1\ldots,1}^{s-l}).
$$
Recall the definitions in \S \ref{sec:2-2} of 
$$
A=(\mbi a_1|\ldots |\mbi a_r)\in M(n+r,r,\Z)
\mbox{ and } 
B=(\mbi b_1|\ldots |\mbi b_s)\in M(n+r,s,\Z),
$$ 
where $\mbi a_i,\mbi b_i$ are column vectors in $\Z^{n+r}$.
Define 
$$
A_\tau=(\mbi b_{s_\tau+1}|\ldots|\mbi b_{l}|\mbi a_1|\ldots |\mbi a_r)\in M(n+r,r+l-s_\tau,\Z)
$$
and 
$$
B=(\mbi b_1|\ldots |\mbi b_{s_\tau})\in M(n+r,s_\tau,\Z).
$$
Then the short exact sequence
$$
0\to\Z^{r+l-s_\tau}\stackrel{A_\tau}{\to} \Z^{n+r}\to N(\tau)\to 0
$$
gives a projective resolution of $N(\tau)$, and  $B_\tau$ defines the map $\beta_\tau$ as $B$ defines $\beta$.
By the explicit description  (\ref{G}),
we have an isomorphism 
$
G_{\Sigma_{\tau}}
\cong G_{\tau}.
$
Then the  isomorphism between 
$U_{\Sigma_{\tau}}$ and $V_{\tau}$ becomes   
$G_{\Sigma_{\tau}}\cong G_\tau$-equivariant, and hence 
we get an isomorphism 
$$
\X_{\Sigma_{\tau}}(=[U_{\Sigma_{\tau}}/G_{\Sigma_{\tau}}])\cong [V_{\tau}/G_{\tau}]
$$
of stacks. 
Now put
$$
V(\tau):=\left\lbrace z_{s_\tau+1}=\cdots =z_l=0\text{ in }U\right\rbrace
(\subset \bigcup_{\tau\subset\sigma\in\Delta}U_{\sigma}).
$$
Then we have the following lemma.


\begin{lem}\label{stackisom}
We have an isomorphism $[V_{\tau}/G_{\tau}]\cong [V(\tau)/G]$ of stacks.
In particular, there is a closed embedding 
\begin{equation*}
\iota\colon \X_{\Sigma_{\tau}}\cong [V(\tau)/G]\hookrightarrow \X=[U/G].
\end{equation*}
\end{lem}

\proof
Embeddings $V_{\tau}\subset V(\tau)$, $G_{\tau}\subset G$ gives a morphism 
$\varphi\colon [V_{\tau}/G_{\tau}]\to [V(\tau)/G]$ of stacks.
We show that  $\varphi$ is an isomorphism.

To show that $\varphi$ is essentially surjective, since $[V(\tau)/G]$ is a sheaf, 
it is enough to show that for any object $P$ of $[V(\tau)/G](W)$ over a scheme $W$ there exists an \'{e}tale covering 
$\lbrace W_i\to W\rbrace_{i}$ of $W$ such that $P|_{W_i}$ is in the essential image of 
$\varphi _{W_i}\colon [V_{\tau}/G_{\tau}](W_i)\to [V(\tau)/G](W_i)$
 for any $i$.
First take an \'{e}tale covering such that $P|_{W_i}$ is given by a trivial principal $G$-bundle 
 $W_i\times G\to W_i$ and a $G$-equivariant morphism $\psi$:
$$
V(\tau) \stackrel{\psi}\longleftarrow W_i\times G \longrightarrow W_i.
$$
Furthermore for each $i$, we may assume that there exists 
a cone $\sigma\in\Delta$ satisfying $\tau\subset\sigma$,
such that $\im \psi\subset U_{\sigma}\cap V(\tau)$.

Then by the explicit description (\ref{G}) and (\ref{action2}), we can take an  \'{e}tale cover $W'_i\to W_i$ 
such that the morphism 
$$
W'_i\longrightarrow W_i\stackrel{\psi|_{W_i\times \id_G}}{\longrightarrow}U_{\sigma}\cap V(\tau)
$$
is also decomposed as  
$$
W'_i\longrightarrow U_{\sigma}\cap V_{\tau}{\hooklongrightarrow}U_{\sigma}\cap V(\tau)
$$
 up to the $G$-action on $U_{\sigma}$. 
Hence $P|_{W'_i}$ belongs to the essential image of $\varphi _{W'_i}$.

Note that for the $G$-action on $U$ the subgroup of elements of $G$ keeping $V_{\tau}$ is equal to $G_{\tau}$.
Hence we can see that $\varphi$ is fully faithful. This completes the proof.
\endproof

We call $\X_{\Sigma_{\tau}}$ 
the \emph{toric substack  of $\mc X$ 
corresponding to the cone $\tau$}.
We often denote it by
$$
[ z_{s_{\tau}+1}=\cdots =z_{l}=0].
$$
%
For $\mbi k={}^t(k_1,\ldots,k_{s})\in \Z^{s}$
and $\mbi l={}^t(l_1,\ldots,l_{r})\in \Z^{r}$,
we put 
$$
\mbi k_{\tau}={}^t(k_1,\ldots,
k_{s_{\tau}})\in \Z^{s_{\tau}},\quad
\mbi l_{\tau}={}^t(k_{s_{\tau}+1},\ldots,k_{l},
l_{1},\ldots,l_{r})\in \Z^{l-s_\tau+r}
$$
and $\mbi D_{\tau}={}^t(\iota^{\ast}\mc D_{1},
\ldots, \iota^{\ast}\mc D_{s_{\tau}})$.
By Lemma~\ref{stackisom}, we have
\begin{equation}\label{emb}
\iota_{\ast}\mc O_{\mc X_{\Sigma_\tau}}
(\mbi k_{\tau}\mbi D_{\tau})_{\mbi l_{\tau}} 
=
\mc O_{\mc X}\left(\mbi k\mbi D\right)_{\mbi l}\otimes\iota_*\mc O_{\mc X_{\Sigma_\tau}},
\end{equation}
since we can compute push-forward by the embedding 
$[V(\tau)/G]\hookrightarrow[U/G]$ as in the
proof of \cite[Theorem~9.1]{BH1}.


\subsection{Morphisms between toric DM stacks}\label{sec:2-3}
In this subsection
we consider a morphism between toric DM stacks $\X$ and $\X'$.

First let us consider the following morphism of triangles of the derived category of $\Z$-modules: 
%
\begin{equation}\label{mor}
\xymatrix{ 
\cdots \ar[r]&\Z^s\ar[r]^{\beta} \ar[d]_{\gamma_1=C}&N\ar[d]^{\gamma_2}\ar[r]&\Cone (\beta)\ar[d]^{\gamma_3}\ar[r]&\cdots\\
\cdots \ar[r]&\Z^{s'} \ar[r]^{\beta'}                       & N'\ar[r]      &                   \Cone (\beta')\ar[r]&\cdots,
}
\end{equation}
where $\gamma_1$ is a matrix $C=(c_{i,j})\in M(s',s,\Z)$.
We assume that
\begin{itemize}
\item[(T1)]
$\gamma_2$ induces a map $\Bar{\gamma}_2\colon (\Bar{N},\Delta)\to (\Bar{N'},\Delta')$ of fans,
 that is, for every cone $\sigma\in\Delta$, there exists a cone $\sigma'\in\Delta'$ satisfying
${\gamma}_{2,\R}(\sigma)\subset\sigma'$.
\end{itemize}
Take $j$ satisfying $(\gamma_2\circ\beta)_\R (\mbi f_j)\in \sigma'$  
for some cone $\sigma'\in \Delta'$. 
We furthermore assume that
\begin{itemize}
\item[(T2)]
we have  
\begin{equation*}
\gamma_1(\mbi f_j)=\sum _{i=1}^{s'}c_{i,j}\mbi f'_i\in\bigoplus_{\beta'_{\R}(\mbi f_i')\in\sigma'}\Z_{\ge 0}\mbi f'_i.
\end{equation*}
\end{itemize}
As its consequence, we know that every $c_{i,j}$ is non-negative. 
Under these assumptions we construct a morphism $\varphi\colon \X=[U/G]\to\X'=[U'/{G'}]$ of DM stacks as follows.

First consider the morphism 
\begin{equation*}
\gamma _{1,\C^*}:=\gamma _1\otimes_\Z \id_ {\C^*} \colon (\C^*)^s\to (\C^*)^{s'}
\quad (z_j) \mapsto \left(z_i'\right)=\left( \prod_jz_j^{c_{i,j}}\right).
\end{equation*}
We can naturally extend it to a morphism $\C^s\to {\C}^{s'}$, which is also denoted by $\gamma _{1,\C^*}$.
Then ${\gamma _{1,\C^*}}$ induces the morphism $U\to U'$ as follows:
take a point 
$p=(p_j)\in U_\sigma= \C^s\backslash \{z_\sigma =0\}$
for some cone $\sigma \in\Delta$. If $p_j=0$, 
then $j$ satisfies that $\beta_{\R}(\mbi f_j)\in \sigma$.
Take a cone $\sigma'$ such that $\gamma_{2,\R}(\sigma)\subset \sigma'$. 
Then the second assumption above implies that $c_{i,j}=0$
for any $i$ with $\beta'_\R (\mbi f'_i) \notin \sigma'$. 
This means that ${\gamma _{1,\C^*}}(p)\in U'_{\sigma'}$.
We again denote this morphism by
$$
\gamma _{1,\C^*}\colon U\to U'.
$$

On the other hand, the map $\gamma_3$ defines 
a group homomorphism 
$\rho=\Hom (H^1(\gamma_3^\star),\C^*)$
$$
\rho\colon G=\Hom_\Z (H^1(\Cone (\beta) ^\star),\C^*)
\to{G'}=\Hom _\Z(H^1(\Cone (\beta ') ^\star),\C^*).
$$
Then we have the following commutative diagram
\begin{equation}\label{dai}
\xymatrix{ G\times U \ar[r]\ar[d]_{\rho\times \gamma _{1,\C^*}}
&U\ar[d]^{\gamma_{1,\C^*}}\\
{G'}\times U' \ar[r]& U',}
\end{equation}
where the horizontal arrows are defined by the actions of $G$ 
on $U$ and $G'$ on $U'$.
This diagram determines a morphism 
$$
\varphi\colon\X=[U/G]\to \X'=[U'/{G'}]
$$
of DM stacks.


\begin{rem}\label{rem:stackyfan}
By \cite[Theorem 1.2]{Iw2}, we see that 
giving torus equivariant morphisms 
between toric DM orbifolds is equivalent to 
giving homomorphism $\gamma_1,\gamma_2$ 
as in (\ref{mor}) satisfying assumptions (T1) 
and (T2). 
\end{rem}

Let us take projective resolutions of $N$ and $N'$;
$$ 
0\to \Z^r\stackrel{A}{\to}\Z^{n+r}\to N\to 0 \mbox{ and }
0\to \Z^{r'}\stackrel{A'}{\to}\Z^{n'+r'}\to N'\to 0
$$
for $A\in M(n+r,r)$ and $A'\in M(n'+r',r')$.
The map  $\gamma_2$ is determined by matrices 
$$
D=(d_{i,j})\in M(n'+r',n+r,\Z),\quad  E=(e_{i,j})\in M(r',r,\Z),
$$
which make the diagram in the right square commutative 
(but not necessarily in the left):
$$
\xymatrix{ \Z^{s}\ar[d]_{C} \ar[r]^{B}&
\Z^{n+r}\ar[d]^{D}\ar@{}[dr]|\circlearrowleft&\Z^{r}\ar[d]^{E} \ar[l]_{A} \\
\Z^{s'} \ar[r]_{B'}&\Z^{n'+r'}& \Z^{r'} \ar[l]^{A'}.}
$$
Here $B\in M(n+r,s)$ and $B'\in M(n'+r',s')$ give lifts of the maps
 $\beta\colon \Z^s\to N$ and $\beta'\colon \Z^{s'}\to N'$  respectively.
By the commutativity of (\ref{mor}), we have a homotopy homomorphism $F=(f_{i,j})\colon \Z^s\to\Z^{r'}$ 
such that $DB-B'C=A'F$.
The maps $\gamma_{1,\C^*}$ and $\rho$ are described as follows: 
\begin{equation}\label{eqn:UG}
\begin{split}
&\gamma_{1,\C^*} \colon U\to U'\quad (z_j)\mapsto \left(z_i'\right)=\left( \prod_jz_j^{c_{i,j}}\right)\\
&\rho \colon G\to {G'}\quad (t_j)\mapsto \left(t_i'\right),
\end{split}
\end{equation}
where we put 
$$
t_i':=\prod_{j=1}^{s} t_j^{c_{i,j}}\text{ for }i=1,\ldots,s'
\quad \text{ and } \quad
t'_{s'+k}:=\prod_{j=i}^st_j^{f_{k,j}}\prod_{l=1}^r t_{s+l}^{e_{k,l}}\text{ for }k=1,\ldots,r',
$$
and $z_j^{c_{i,j}}=1$ when $z_j=c_{i,j}=0$. 
Consequently we have
\begin{equation}\label{eqn:pullback}
\varphi^*\mc O_{\X'}(\mc D_i')=\mc O_{\X}(\sum_{j}c_{i,j}\mc D_j) 
\text{ and } \varphi^*\mc O_{\X',{\mbi g'^\ast_k}}
=\mc O_{\X}(\sum_{j}f_{k,j}\mc D_j)_{\sum_le_{k,l}\mbi g_l^{\ast}}.
\end{equation}
For the homogeneous coordinate ring $S=\C[z_1,\ldots,z_s]$ 
(resp. $S'=\C[z'_1,\ldots,z'_{s'}]$) of $\mc X$ (resp. $\mc X'$), 
we define the map 
$$
\varphi^{\sharp}\colon S'\to S \quad (z_i')\mapsto \prod_jz_j^{c_{i,j}}
$$
so that $\varphi^{\sharp}(z'^{\mbi k'})=z^{\gamma _1^*(\mbi k')}$. 
In particular, we have  
$\varphi^{\sharp}(S'_{\chi'})\subset S_{\rho^{\vee}(\chi')}$,
where 
$$
\rho^{\vee}\colon {G'}^{\vee}\to G^{\vee}
$$
is the $\C^*$-dual map $\Hom _\Z(\rho,\C^*)$ of $\rho$. 
We have the following commutative diagram:
\begin{equation}\label{eqn:dual}
\xymatrix{ 
\cdots\ar[r]&\Hom_{\Z}(N,\Z) \ar[r]   & \Z^{s}\ar[r]^{\cl\qquad}  & \DG(\beta)=G^\vee\ar[r]& \Ext_{\Z}^1(N,\Z) \ar[r]&\cdots\\
\cdots\ar[r]&\Hom_{\Z}(N',\Z) \ar[u]^{\gamma_2^*} \ar[r]& \Z^{s'}\ar[u]_{\gamma_1^*={}^tC}\ar[r]^{\cl\qquad}   
 & \DG(\beta ')=G'^\vee \ar[u]_{\rho^\vee}\ar[r]&\Ext_{\Z}^1(N',\Z)\ar[u]^{\gamma_2^*[1]}\ar[r]&\cdots .}
\end{equation}

If a ${G'}^{\vee}$-graded $S'$-module $M=\bigoplus_{\chi'\in {G'}^\vee}M_{\chi'}$ gives a coherent sheaf $\mc F$ on $\X'$
by  (\ref{grS}), 
then a $G^{\vee}$-graded $S$-module $M\otimes_{S'}S$ corresponds 
to the pull-back $\varphi^{\ast}\mc F$, where the grading is given by 
$$
(M\otimes_{S'}S)_{\chi}
=\sum_{\substack{\eta'\in {G'}^{\vee}}}M_{\eta'}\otimes_{S'} S_{\chi-\rho^{\vee}(\eta')}
$$
for each $\chi\in G^{\vee}$.

On the other hand, a $G^{\vee}$-graded $S$-module 
$L=\bigoplus_{\chi\in G^{\vee}}L_{\chi}$
gives a coherent sheaf $\mc G$ on $\X$, 
and the sheaf  $\varphi_{\ast}\mc G$ corresponds  
a ${G'}^{\vee}$-graded $S'$-module ${}_{S'}N$ 
whose grading is given by 
%
$$
({}_{S'}L)_{\chi '}=L_{\rho^{\vee}(\chi')}
$$
for $\chi'\in {G'}^{\vee}$.
Here $S'$-module structure of ${}_{S'}N$ is given by $\varphi^{\sharp}$.
This $S'$-module structure is compatible 
with the ${G'}^{\vee}$-grading of ${}_{S'}L$, 
since $\varphi^{\sharp}(S'_{\chi'})\subset S_{\rho^{\vee}(\chi')}$. 

The functor ${}_{S'}(-)\colon\gr S/\tor S\to \gr S'/\tor S'$ 
is the right adjoint functor
of $(-)\otimes_{S}S'$. 
Since $\varphi_{\ast}\colon \Coh\mc X\to \Coh \mc X'$ 
is the right adjoint functor of
$\varphi^{\ast}\colon \Coh\mc X'\to \Coh \mc X$, 
both functors ${}_{S'}(-)$ and $\varphi_{\ast}$ 
must coincide by the correspondence in (\ref{grS}). 
Hence
for any $\mc G\in\Coh\mc X$, we have a natural isomorphism 
\begin{equation}\label{pfwd}
H^0(U',\varphi_{\ast}\mc G) \cong
\bigoplus_{\chi'\in {G'}^{\vee}}
H^0(U,\mc G)_{\rho^{\vee}(\chi')}
\end{equation}
in $\gr S'/\tor S'$.


\begin{lem}\label{van}
Take a morphism $\varphi\colon\mc X\to\mc X'$ determined by matrices 
$C=(c_{i,j})\in M(s',s,\Z), D\in M(n'+r',n+r,\Z)$ and $E\in M(r',r,\Z)$ as above.
Then we have the following.
\begin{enumerate}
\item
Suppose that for each $i$ there exists $j$ such that $c_{i,j}>0$, and entries of the $j$-th column vector of $C$ are all zero except $c_{i,j}$.
Furthermore,  in the diagram (\ref{eqn:dual}), we assume that $\gamma_2^*$ is surjective, 
and $\gamma_2^*[1] $ is injective. 
Then we have $\varphi_{\ast}\mo_{\X}=\mo_{\X'}$.
\item 
Assume that
$s=s'$, $C=(c_i\delta_{i,j})$ for $c_i\in \Z_{> 0}$, and that the map of fans 
$\bar{\gamma_2}\colon (\Bar{N},\Delta) \to (\Bar{N}',\Delta')$ is an isomorphism. 
Then we have $\R^{i}\varphi_{\ast}\mc F=0$ for any coherent sheaf $\mc F$ on $\mc X$ and $i>0$.
\end{enumerate}
\end{lem}

\proof
(i) %
For any $\chi'\in {G'}^{\vee}$, we consider sets 
$$
\mathfrak{S}'=\lbrace\mbi k'\in\Z^{s'}_{\ge0}\mid \cl(\mbi k')=\chi'\rbrace, \quad\mathfrak{S}=\lbrace\mbi k\in\Z^s_{\ge 0}\mid \cl(\mbi k)=\rho^{\vee}(\chi')\rbrace.
$$ 
By the assumption on $C$, we see that ${}^tC\colon \Z^{s'}\to \Z^{s}$ is injective, and $({}^tC)^{-1}(\Z^s_{\ge 0})= \Z^{s'}_{\ge 0}$.
When $\mathfrak{S}'\neq \emptyset$, in the diagram (\ref{eqn:dual}), combining the surjectivity of $\gamma_2^*$ and the injectivity of ${}^tC$, we conclude that the map ${}^tC$ induces a bijection between $\mathfrak{S}'$ and $\mathfrak{S}$.

On the other hand, the injectivity of the map  $\gamma_2^*[1]$ is equivalent to 
the condition ${\rho^{\vee}}^{-1}(\cl (\Z^s))=\cl(\Z^{s'})$. So
if $\mathfrak{S}'=\emptyset$, then $\mathfrak{S}=\emptyset$.
Hence we obtain  an isomorphism $S'\cong {}_{S'}S$ by  the map $\varphi^{\sharp}$. 
This completes the proof. 

(ii) For any $\sigma\in \Delta'=\Delta$, we have $\varphi^{\sharp}(z'_{\sigma})=z^{\mbi k}z_{\sigma}$ 
for some $\mbi k\in\Z_{\ge 0}^{s}$, where $\mbi k=(k_i)$ satisfies that $k_i=0$ if $\beta_{\R}(\mbi f_i)\in \sigma$.
Hence if we take a $G^{\vee}$-graded $S$-module $L\in\tor S$, then ${}_{S'}L$ also belongs to $\tor S'$.
We see that an exact sequence in the abelian category $\gr S/\tor S$ is given by an exact sequence of graded $S$-modules
$$
0\to L_1\stackrel{f}{\to} L_2\stackrel{g}{\to} L_3
$$ 
such that $\coker g$ belongs to $\tor S$.
Hence $\varphi_{\ast}\colon \Coh \mc X\to\Coh \mc X'$ is an exact functor,
which implies the assertion. 

\endproof


\section{Toric DM stacks vs. toric DM orbifolds}\label{sec:4}
Let us introduce two kinds of \emph{root stacks}, 
important notions in this note.
In this section, we consider the toric DM stack $\mc X=\mc X_{\Sigma}$ 
associated to a stacky fan $\Sigma=(\beta,\Delta)$ in 
$$
N:=\Z^{n}\oplus\bigoplus_{i=1}^{r} \Z_{a_i}
$$
for some $a_i\in\Z_{>0}$.
As in \S \ref{sec:2-2} we take matrices 
$A=\begin{pmatrix}O_{n\times r}\\
\diag(a_1,\ldots,a_r)\end{pmatrix}\in M(n+r,r,\Z)$ 
and $B=(b_{i,j})\in M(n+r,s,\Z)$
such that $N\cong \coker A$ and $\beta$ is defined by 
$
\beta\colon \Z^s\stackrel{B}{\to}\Z^{n+r}\twoheadrightarrow \coker A.
$


\subsection{Root stacks of line bundles on toric DM stacks}
\label{rootl}
For $r'>r$, $\mbi e={}^t(e_1,\ldots,e_{r'-r})\in\Z^{r'-r}_{>0}$ and 
a collection $\mbi L={}^t(\mc L_1,\ldots,\mc L_{r'-r})$ of line bundles
on $\X$, we consider the ($\mbi e$-th) \emph{root stack} 
$\sqrt[\mbi e]{\mbi L/\mc X}$
of $(\mc X,\mbi L)$ (\cite[1.3.a]{FMN}),
that is, the fiber product:
$$
\xymatrix{
\sqrt[\mbi e]{\mbi L/\mc X}\ar[d]\ar[r]\ar@{}[dr]|\square
&(\mc B\C^{\ast})^{r'-r}\ar[d]^{\wedge \mbi e}\\
\mc X\ar[r]^{\mbi L}&(\mc B\C^{\ast})^{r'-r}.}
$$
Here $\mc B\C^{\ast}$ is the quotient stack of a point by 
the trivial $\C^{\ast}$-action.
By abuse of notation, we use symbols $\mbi L\colon\mc X\to 
(\mc B\C^{\ast})^{r'-r}$ and $\wedge\mbi e \colon 
(\mc B\C^{\ast})^{r'-r}\to (\mc B\C^{\ast})^{r'-r}$ to denote the  morphisms
induced by $\mbi L$ and 
$\wedge \mbi e\colon (\C^{\ast})^{r'-r}\to (\C^{\ast})^{r'-r}$ respectively.

The root stack $\sqrt[\mbi e]{\mbi L/\X}$ is constructed as 
a toric stack in the following way.
We take $(\mbi k_i,\mbi l_i)\in \Z^{s}\oplus\Z^{r}$ such that
$\mc L_i=\mo_\X(\mbi k_i\mbi D)_{\mbi l_i}$ in $\Pic \X$ for 
$i=1,\ldots,r'-r$
and matrices
$$
A'=
\begin{pmatrix}
A&0&\hdots&0\\
-{}^t\mbi l_1&e_1&&\\
\vdots&&\ddots&\\
-{}^t\mbi l_{r'-r}&&&e_{r'-r}
\end{pmatrix}
\in M(n+r',r',\mbi Z),\quad
B'=
\begin{pmatrix}
B\\
-{}^t\mbi k_1\\
\vdots\\
-{}^t\mbi k_{r'-r}
\end{pmatrix}\in
M(n+r',s,\Z).
$$
We define an abelian group $N'=\coker A'$
and a stacky fan $\Sigma'=(\beta',\Delta)$ in $N'$ 
by the map 
$\beta'\colon \Z^s\stackrel{B'}{\to}\Z^{n+r'}\twoheadrightarrow \coker (A')$.

We have a commutative diagram
$$
\xymatrix{ \Z^s\ar[r]^{B'}\ar@{=}[d]_{C=\id_{\Z^s}}
&\Z^{n+r'}\ar[d]^{D}
&\ar[l]_{A'}\Z^{r'}\ar[d]^{E}\\
\Z^s\ar[r]_{B}& \Z^{n+r}&\ar[l]^{A}\Z^r,
} 
$$
where $D$ and $E$ is the projections.
Then we have a morphism
$
\Psi\colon \mc X'=\mc X_{\Sigma'}\to \X
$
by the result in \S \ref{sec:2-3}, and 
$\Psi$ satisfies that 
$$
\Psi ^*\mc O_{\mc X}(\mbi k\mbi D)_{\mbi l}=
\mc O_{\mc X'}(\mbi k\mbi D')_{\mbi l}
$$
for any $(\mbi k,\mbi l)\in\Z^{s}\oplus\Z^{r}$ by (\ref{eqn:pullback}),
where we identify $\mbi l$ with the element in $\Z^{r'}$ 
by the injection $\Z^r\to \Z^r\oplus \Z^{r'-r}\cong \Z^{r'}$.

In particular, by the choice of $A'$ and $B'$, we have 
isomorphisms 
$$
\Psi ^*\mc O_{\mc X}(\mbi k_i\mbi D)_{\mbi l_i}\cong
(\mo_{\X',\mbi g_{r+i}'^{\ast}})^{\otimes e_i}
$$
for $i=1,\ldots,r'-r$. 
Hence $\Psi$ and 
$$
\mbi L'={}^t(\mc O_{\X',\mbi g_{r+1}'^{\ast}},\ldots,
\mc O_{\X',\mbi g_{r'}'^{\ast}})\colon \X'\to (\mc B\C^{\ast})^{r'-r}
$$
gives a morphism $\X'\to\sqrt[\mbi e]{\mbi L/\X}$.
By \cite[Theorem 2.6]{Pe},
we see that this is an isomorphism.

We have the following theorem.


\begin{thm}\label{thm:orb-vs-stack}
For $\X'=\sqrt[\mbi e]{\mbi L/\mc X}$, we have an equivalence
$$
\Coh \X'\cong \bigoplus _{\substack{(l_i)\in \Z^{r'-r}\\
0\le l_i<e_i}}
(\Psi^{\ast}\Coh \X)\otimes 
\mc O_{\mc X',\sum_i l_i\mbi g_{r+i}'^{\ast}},
$$
where in the right hand side the symbol $\bigoplus$ means that if $(l_i)\neq (m_i)\in \Z^{r'-r}$, then we have 
$$ 
\Hom(E,F)=0
$$
for any $E\in(\Psi^{\ast}\Coh \X)\otimes 
\mc O_{\mc X',\sum_i l_i\mbi g_{r+i}'^{\ast}}$ and $F\in (\Psi^{\ast}\Coh \X)\otimes 
\mc O_{\mc X',\sum_i m_i\mbi g_{r+i}'^{\ast}}$.
\end{thm}

\begin{proof}
Applying \cite [Lemma 4.1]{IU} repeatedly, 
we obtain the assertion.
\end{proof}


\subsection{The rigidification of toric DM stacks}\label{subsec:orb&stack}
We define the \emph{rigidification} $\X^{\text{rig}}$ of 
$\X$ according to \cite{FMN} as follows.
Let us take a stacky fan $\Sigma^{\text{rig}}=
(\beta^{\text{rig}},\Delta)$ in $\Bar{N}=\Z^n$, where 
we define $\beta^{\text{rig}}\colon \Z^s\to \bar{N}$ to be 
a composition of $\beta$ and a natural surjection $N\to \bar{N}$, 
and consider the toric DM orbifold 
$\mc X^{\text{rig}}:=\X_{\Sigma^{\text{rig}}}$.

We write $B=(b_{i,j})$ for $b_{i,j}\in\Z$ and 
put $\mc L_{i}=\mo_{\X^{\rig}}(-\sum_jb_{n+i,j}\mc D_{j})$
for $i=1,\ldots,r$. 
Then the stack $\X$ is isomorphic to 
the ${}^t(a_1,\ldots,a_r)$-th root stack of 
$(\X^{\rig},(\mc L_1,\ldots,\mc L_r))$.
In particular, we have a morphism
$
\Psi\colon \X \to \X^{\text{rig}}
$
as in \S \ref{rootl}.
We call this $\Psi$ the
\emph{rigidification morphism}. 
By applying Theorem \ref{thm:orb-vs-stack} to $\Psi$, we obtain the following corollaries:


\begin{cor}\label{cor:orb-vs-stack}
Let $\X$ be a toric DM stack and $\X^{\text{rig}}$ its rigidification.
Then we have an equivalence
$$
\Coh \X\cong \bigoplus _{\substack{(l_i)\in \Z^{r}\\
0\le l_i<a_i}}
(\Psi^{\ast}\Coh \X^{\rig})\otimes 
\mc O_{\mc X,\sum_i l_i\mbi g_{i}^{\ast}}.
$$
%
\end{cor}

The following must be well-known to specialists. 
It is used in the proof of Lemma \ref{lem}.


\begin{cor}\label{cor:finite-length}
Let $\X$ be a toric DM stack. 
Then $\Coh \X$ has a finite homological dimension, 
namely any coherent sheaf on $\X$ has a locally 
free resolution of finite length.
\end{cor}

\begin{proof}
$\Coh \X^{\rig}$ has a finite homological dimension 
by the proof of \cite[Theorem 4.6]{BH1}. 
Hence the assertion follows from 
Corollary \ref{cor:orb-vs-stack}.
\end{proof}

The following are stacky generalizations of the  results for toric DM orbifolds in \cite{Ka} 
and \cite[Theorem~7.3]{BHu}.
 

\begin{cor}\label{cor:kawamata}
\begin{enumerate}
\item
Let $\X$ be a toric DM stack. 
Then $D^b(\X)$ has a full exceptional collection 
consisting of coherent sheaves.
\item
Let $\X$ be a two dimensional toric DM stack 
whose rigidification $\mc X^{\text{rig}}$ has
an ample anti-canonical divisor.  
Then $D^b(\X)$ has a full strong exceptional collection 
consisting of line bundles.
\end{enumerate}
\end{cor}

\begin{proof}
The assertion (i) (respectively (ii)) directly follows from 
Corollary \ref{cor:orb-vs-stack} and  the result in \cite{Ka} (resp. \cite[Theorem~7.3]{BHu}).
\end{proof}


\subsection{Root stacks of effective Cartier divisors on toric DM stacks}\label{sec:2-5}
For positive integers $c_1,\ldots,c_s$, we define another 
stacky fan $\Sigma''=(\beta'',\Delta)$ in $N$ 
by replacing $\beta$ in $\Sigma$ with 
$$
\beta''\colon \Z^s\stackrel{BC}{\to}\Z^{n+r}\twoheadrightarrow \coker A,
$$
where we put $C=\diag (c_1,\ldots,c_s)\in M(s,s,\Z)$.
We call $\mc X''=\mc X_{\Sigma''}$ the ($\mbi c$-th) 
\emph{root stack} of $(\mc X,\mbi D)$ 
(\cite[1.3.b]{FMN}) and denote it by 
$$
\sqrt[\mbi c]{\mbi D/\mc X},
$$ 
where we define $\mbi c={}^t(c_1,\ldots,c_s)$ 
and $\mbi D={}^t(\mc D_1,\ldots,\mc D_s)$ is the collection
of toric divisors on $\X$ as in \S \ref{sec:2-2}.
We have a commutative diagram
$$
\xymatrix{ \Z^s\ar[r]^{BC}\ar[d]_{C}
&\Z^{n+r}\ar@{=}[d]^{D= \id_{\Z^{n+r}}}&\ar[l]_{A}\Z^r\ar@{=}[d]^{E= \id_{\Z^r}}\\
\Z^s\ar[r]_{B}& \Z^{n+r}&\ar[l]^{A}\Z^r.
} 
$$
Then we have a morphism 
$
\varphi\colon \mc X''=\sqrt[\mbi c]{\mbi D/\mc X}\to \X
$
by the result in \S \ref{sec:2-3}.
We call $\varphi$ the 
\emph{root construction morphism}.
We obtain the following by Lemma~\ref{van} (i), (ii).


\begin{lem}\label{lem:rootc-vanish}
$\mathbb R\varphi_{\ast}\mo_{\mc X''}=\mo_{\mc X}$.
\end{lem}

Next assume furthermore that $\mc X=\mc X_{\Sigma}$ is an orbifold.
Then in the notation in \S \ref{sec:2-1}, $\X$ is denoted by 
$
\X\bigl(X,\sum_{i=1}^s b_iD_i\bigr)
$
for some $b_i\in \Z_{>0}$ and 
the coarse moduli space $X=X_{\Delta}$ of $\X$.  
Then $\mc X$ is realized as a root stack 
$\sqrt[\mbi b]{\mbi D/ \X^{\text{can}}}$ over the canonical stack
$\mc X^{\text{can}}=\X\bigl(X,\sum D_i\bigr)$
for $\mbi b={}^t(b_1,\ldots,b_s)$.
Additionally if $X$ is smooth, then $\mc X$ is 
a root stack $\sqrt[\mbi b]{\mbi D/ X }$ over $X$ 
by Remark \ref{rem:cox}. 



\section{Frobenius push-forward for toric DM stacks}\label{sec:3}

\subsection{Generators}
Let us begin this section with important definitions.
Suppose that  $\X$ is a smooth complete DM stack.


\begin{defn}\label{def:generator}
For a subset  $\mc S\subset D^b(\X)$,  $\langle \mc S \rangle$ denotes the smallest full triangulated subcategory containing $\mc S$ 
of $D^b(\X)$ such that $\langle \mc S \rangle$ is stable under taking direct summands and direct sums. 

We say that  $\mc S$ is a \emph{generator} of $D^b(\X)$, or $\mc S$ \emph{generates} $D^b(\X)$
 if $\langle \mc S \rangle=D^b(X)$.
If $\mc S$ consists of a single element $\alpha$, we just say that $\alpha$ is a generator of $D^b(\X)$.
\end{defn}


\begin{defn} 
\begin{enumerate}
\item
An object $\mc E\in D^b(\X)$ is called \emph{exceptional} if it satisfies 
$$
\Hom _{D^b(\X)}^i(\mc E,\mc E)=
\begin{cases}
\C &\mbox{$i=0$}\\
0 &\mbox{otherwise}.
\end{cases}
$$ 
\item
An ordered set $(\mc  E_1,\ldots,\mc E_n)$ 
of exceptional objects is called an \emph{exceptional collection} 
if the following condition holds;
$$\Hom ^i_{D^b(\X)}(\mc E_k,\mc E_j)=0$$ 
for all $k>j$ and all $i$. When we say that a finite set $\mc S$
of objects is an exceptional collection, it means that $\mc S$ 
forms an exceptional collection in an appropriate order. 
\item 
An exceptional collection $(\mc  E_1,\ldots,\mc E_n)$ is called 
\emph{strong} if 
$$\Hom ^i_{D^b(\X)}(\mc E_k,\mc E_j)=0$$ 
for all $k,j$ and $i\ne 0$.
\item
An exceptional collection $(\mc E_1,\ldots,\mc E_n)$ is called 
\emph{full} if 
the set $\{ \mc  E_1,\ldots,\mc E_n\}$ generates $D^b(\X).$ 
\end{enumerate}
\end{defn}

\begin{rem}\label{rem:Grothendieck}
Let  $\X=\X_{\Sigma}$ be a toric DM orbifold associated with a stacky fan $\Sigma=(\Delta, \beta)$.  
If the toric DM stack $\X=\X_{\Sigma}$ has a full exceptional collection consisting of
$n$ exceptional objects, the rank of its Grothendieck group $K(\X)$ is $n$.
Furthermore suppose that $-K_{\X}$ is nef, which is equivalent to the condition that 
 all $\beta(\mbi f_i)$'s lie on the boundary of  the convex hull of all $\beta(\mbi f_i)$'s (see also \S \ref{sec:6} for the notion \emph{nef}).
Then it is proved in \cite[Corollary 2.2 and Proposition 3.1 (iii)]{BH2} that
$$
\rk K(\X)=\rk N! \vol \Delta,
$$
where 
$\vol \Delta$ is the volume of the convex hull of all $\beta(\mbi f_i)$'s. 
\end{rem}


\subsection{Frobenius morphism}\label{subsec:Frobenius}

Below we use the terminology in \S\ref{sec:2-1} and \ref{sec:2-2}.
For a positive integer $m$, we consider the Frobenius morphism 
$F(=F_m)\colon \X\to\X$ induced by 
$$
\wedge m\colon U\to U \text{ and } \wedge m\colon G\to G.
$$ 
Take both of  $\gamma_1\colon \Z^s\to \Z^s$ 
and $\gamma_2\colon N\to N$ in (\ref{mor}) 
as the multiplication maps by $m$.
Then we obtain the \emph{Frobenius morphism} $F$, 
which is actually a generalization of the 
usual  Frobenius morphism.


\begin{lem}\label{lem}
Let $\X$ and $\mc Y$ be toric DM stacks. 
Consider a proper morphism $\varphi\colon \X\to\mc Y$ 
which satisfies $\mathbb R\varphi_{\ast}\mo_{\X}
=\mo_{\mc Y}$.
Then $D^b(\X)=\langle F_{\ast}\mo_{\X}\rangle$ 
implies $D^b(\mc Y)=\langle F_{\ast}\mo_{\mc Y}\rangle$.
\end{lem}

\proof
We see that $\Coh \mc Y$ has a finite homological dimension by Corollary \ref{cor:finite-length}.  
For any object $\mathcal{F}\in D^b(\mc Y)$, by the assumption we have
 $\mathbb L\varphi^{\ast}\mathcal{F}\in\langle F_{\ast}\mo_{\X}\rangle$.
Since we have $\mathbb R\varphi_\ast F_{\ast}\mathcal{O}_{\X}=F_{\ast}\mathbb R\varphi_\ast \mathcal{O}_{\X}=F_{\ast}\mo_{\mc Y}$, we see that $\mathcal{F}\cong\mathbb R\varphi_\ast \mathbb L\varphi^{\ast}\mathcal{F}$ belongs to $\langle F_{\ast}\mathcal{O}_{\mc Y}\rangle$.
\endproof


\subsection{Direct summands of Frobenius push-forward}
We consider the Frobenius morphism 
$$
F\colon \mc X\to \mc X'.
$$ 
Although the target space $\mc X'$ is $\mc X$ itself,  
we use the notation $\mc X'$ so that we distinguish the domain $\mc X$ and the target $\mc X'$. 
Theorem \ref{frobenius} for toric DM stacks generalizes the result for toric varieties in \cite{Ac}. 
A similar proof in \cite{Ac} works for the stacky case, but
we give another proof using (\ref{pfwd}).


\begin{thm}\label{frobenius}
For $\chi \in G^{\vee}$ and $\chi'\in {G'}^{\vee}$, 
 put
$$
m(\chi,\chi'):=\sharp\lbrace\mbi j\in [0,m-1]^s\mid m\chi'=\chi-\cl(\mbi j)\rbrace.
$$ 
Then we have 
$$
F_{\ast}\mc L_{\chi}=\bigoplus_{\chi'\in {G'}^{\vee}}\mc L_{\chi'}^{\oplus m(\chi,\chi')}.
$$
\end{thm}

\proof
We identify $\Coh \mc X$ and $\Coh_{G}U$ by (\ref{eq}).
We consider the graded $S$-module 
$$
H^0(U,\mc L_{\chi})=\bigoplus_{\eta\in G^{\vee}}H^0(U,\mc L_{\chi})_{\eta}
$$
and recall that 
$$
H^0(U,\mc L_{\chi})_{\eta}\cong \bigoplus_{\substack{\mbi i\in\Z^{s}_{\ge 0}\\ \cl(\mbi i)=\chi+\eta}}\C z^{\mbi i},
$$
since the $G$-actions in (\ref{action1}) and (\ref{action2}) induce the action
on $H^0(U,\mc L_{\chi})$ given by $z^{\mbi i}\mapsto \chi_{\Bar{\mbi i}}\chi^{-1}z^{\mbi i}$.
We define a $G'^\vee$-graded $S'$-module $M=\bigoplus_{\eta'\in {G'}^{\vee}}M_{\eta'}$ by 
$$
M_{\eta'}:=H^0(U,\mc L_{\chi})_{m\eta'}
\cong \bigoplus_{\substack{\mbi k\in\Z^{s}_{\ge 0}\\ 
\cl(\mbi k)=\chi+m\eta'}}\C z^{\mbi k}.
$$
By (\ref{pfwd}), the $S'$-module $M$ corresponds to 
$F_{\ast}\mc L_{\chi}$ by the equivalence (\ref{grS}). 

We have a one-to-one correspondence between sets 
$\lbrace \mbi k\in\Z_{\ge 0}^{s}\mid \cl(\mbi k)=\chi+m\eta'\rbrace$ and 
$$
\lbrace (\mbi i,\mbi j)\in\Z_{\ge 0}^{s}\times [0,m-1]^s\mid m(\cl(\mbi i)-\eta')=\chi-\cl(\mbi j) \rbrace
$$ 
given by $\mbi k=m\mbi i+\mbi j$ such that $\mbi i=\lfloor\frac{\mbi k}{m}\rfloor$ and $\mbi j\in [0,m-1]^s$.
Thus for each $\eta'\in G'^{\vee}$, 
we have isomorphisms between components of 
$G'^\vee$-graded $S'$-modules
\begin{equation*}
M_{\eta'}
\cong \bigoplus_{\substack{\mbi j\in [0,m-1]^s}}
\left(
\bigoplus_{\substack{\mbi i\in\Z^s_{\ge 0}\\ 
\\ m(\cl(\mbi i)-\eta')=\chi-\cl(\mbi j)}}
\C {z}^{m\mbi i+\mbi j}
\right)
\cong  
\bigoplus_{\substack{\chi'\in {G'}^{\vee}}} 
\left(\bigoplus_{\substack{\mbi j\in [0,m-1]^s\\
m\chi'=\chi-\cl(\mbi j)}}
\left(\bigoplus_{\substack{\mbi i\in\Z^s_{\ge 0}\\ \cl(\mbi i)=\chi'+\eta'}}
\C {z'}^{\mbi i}z^{\mbi j}\right)\right).
\end{equation*}
To obtain the second isomorphism, we put $\chi'=\cl(\mbi i)-\eta'$.
The last component 
$\bigoplus_{\substack{\mbi i\in\Z^s_{\ge 0}\\ 
\cl(\mbi i)=\chi'+\eta'}}\C {z'}^{\mbi i}z^{\mbi j}$
is isomorphic to 
$H^0(U',\mc L_{\chi'})_{\eta'}$.
Hence, summing up all $\eta'\in G'^{\vee}$, 
we have an isomorphism 
$$
M \cong \bigoplus_{\substack{\chi'\in {G'}^{\vee}}} 
\bigoplus_{\substack{\mbi j\in [0,m-1]^s\\
m\chi'=\chi-\cl(\mbi j)}}
H^{0}(U',\mc L_{\chi'})
$$
of $S'$-modules.
Now the desired isomorphism
$$
F_{\ast}\mc L_{\chi}\cong 
\bigoplus_{\chi'\in {G'}^{\vee}}
\bigoplus_{\substack{\mbi j\in [0,m-1]^s\\
m\chi'=\chi-\cl(\mbi j)}}
\mc L_{\chi'}
= \bigoplus_{\substack{\chi'\in {G'}^{\vee}}} 
\mc L_{\chi'}^{\oplus m(\chi,\chi')}
$$
follows from (\ref{grS}).
\endproof

We denote by $\mk D_{\X}(\mc L_{\chi})$ 
the set of isomorphism classes of direct summands 
of $F_\ast \mc L_{\chi}$.
For $\chi\in G^\vee$ and $\chi'\in G'^\vee$, 
there exist some $\mbi k,\mbi k'\in \Z^s$ and 
$\mbi l,\mbi l'\in \Z^r$ such that 
$\mc L_{\chi}=\mo_{\X}(\mbi k\mbi D)_{\mbi l}$ and 
$\mc L_{\chi'}=\mo_{\X'}(\mbi k'\mbi D)_{\mbi l'}$.
Then we have $\chi=\chi_{\overline{\mbi k+\mbi l}}$ and 
$\chi'=\chi'_{\overline{\mbi k'+\mbi l'}}$ as in \S\ref{sec:2-2}.
If $m(\chi,\chi')>0$ then we have 
$$
m\chi'_{\overline{\mbi k'+\mbi l'}}=\chi_{\overline{\mbi k+\mbi l}}-\cl (\mbi j)
$$ 
for some $\mbi j\in \Z^s\cap[0,m-1]^s$, and hence
there exists an element $\mbi u\in\Z^{n+r}$ such that we have
$$
(\mbi k+{}^tB\mbi u)\oplus (\mbi l+{}^tA\mbi u)=(m\mbi k'+\mbi j)\oplus m\mbi l' \quad \mbox{ in } \Z^s\oplus\Z^r.
$$
Thus we have $\mbi k'=\lfloor \frac{\mbi k + {}^tB\mbi u}{m}\rfloor$ 
in $ \Z^s$, $\mbi l'=\frac{\mbi l+{}^tA\mbi u}{m}$ in $\Z^r$ and
\begin{equation}\label{eqn:Thomsen}
\mk D_{\X}(\mc L_{\chi}) =\left\lbrace\mo_\X\left(\lfloor \frac{\mbi k+{}^tB\mbi u}{m}\rfloor{\mbi D}\right)_{\frac{\mbi l+{}^tA\mbi u}{m}}\
 \Bigr|\ \mbi u\in\Z^{n+r}\cap[0,m-1]^{n+r}, \frac{\mbi l+{}^tA\mbi u}{m}\in\Z^r\right\rbrace.
\end{equation}
When $N$ is torsion free, we have $r=0$ and $N=\Z^n$. 
For each $i=1,\ldots,s$, we take a primitive generator $\mbi v_i\in N$ of $\rho_i$ and positive integer $b_i$ such that $\beta(\mbi f_i)=b_i\mbi v_i$.
We have 
$${}^tB\mbi u={}^t\left((\mbi u,b_1\mbi v_1),\ldots,(\mbi u,b_s\mbi v_s)\right),$$ 
where $(\mbi u,b_i\mbi v_i)\in \R$ is just the dot product on $\R^n$.
Consequently, (\ref{eqn:Thomsen}) is a generalization of Thomsen's result \cite{Th} for
smooth toric varieties.

The set $\mk D_\X(\mc O_\X)$ is stabilized for   sufficiently divisible integers $m$ in $F=F_m$, 
and denote this set by $\mk D_\X$.
By the above result, we have  
\begin{equation}\label{formula}
\mk D_\X 
=\left\lbrace\mo_\X\left(\lfloor {}^tB\mbi u\rfloor{\mbi D}\right)_{{}^tA\mbi u}\ \Bigr|\ \mbi u\in[0,1)^{n+r}, {}^tA\mbi u\in\Z^r\right\rbrace.
\end{equation}

\subsection{Reduction to toric orbifolds}\label{subsec:rtto}
We reduce the proof of Main theorem to the orbifold's case.
First take an arbitrary toric DM stack $\X$ and 
we use the notation in \S \ref{sec:4}.
Put $a:=a_1\cdots a_r$ and $\mbi a:=(a,\ldots,a)\in \Z^s$.
For simplicity, denote the root stack 
$\sqrt[\mbi a]{\mbi D/\mc X}$ by $\mc X'=\X_{\Sigma'}$
and consider the rigidification 
$\X'^{\text{rig}}$ of $\mc X'$. 
Then $\beta'$ is defined by the map
$$
\Z^s\stackrel{\begin{pmatrix}aB_1\\\mbi 0\end{pmatrix}}{\to}
\Z^{n+r}\twoheadrightarrow \coker A,$$
where $B_1=(b_{i,j})\in M(n,s,\Z)$.

Here we identify toric divisors on 
$\mc X'$ and $\mc X'^{\text{rig}}$, 
and denote by the same symbol 
$\mbi D'={}^t(\mc D'_1,\ldots,\mc D'_s)$.
For the rigidification morphism $\Psi'\colon
\X'\to\X'^{rig}$,
by (\ref{eqn:pullback}), we have
$\Psi'^{\ast}\mo_{\mc X'^{\text{rig}}}(\mbi k\mbi D')\cong
\mo_{\X'}(\mbi k\mbi D')$ for any $\mbi k\in \Z^s$.
Furthermore by (\ref{formula}), we have
\begin{equation*}
\begin{split}
\mk D_{\X'}
&=\left\lbrace\mo_{\X'}\left(\lfloor a{}^tB_1\mbi u_1\rfloor{\mbi D'}\right)_{\diag(a_1,\ldots,a_r)\mbi u_2}\ \Bigr|\ 
\mbi u_1\in[0,1)^{n}, \mbi u_2\in[0,1)^{r}, \diag(a_1,\ldots,a_r) \mbi u_2\in\Z^r\right\rbrace\\ 
&=\left\lbrace\mo_{\X'}\left(\lfloor a{}^tB_1\mbi u_1\rfloor{\mbi D'}\right)_{\mbi l}\ \Bigr|\ 
\mbi u_1\in[0,1)^{n}, \mbi l\in \Z^r\cap \prod_{i=1}^r[0,a_i)\right\rbrace,\\ 
\mk D_{\mc X'^{\text{rig}}}
&=\left\lbrace\mo_{\mc X'^{\text{rig}}}\left(\lfloor a{}^tB_1\mbi u_1\rfloor{\mbi D'}\right)\ \Bigr|\ 
\mbi u_1\in[0,1)^{n}\right\rbrace.
\end{split}
\end{equation*}
Hence we have $\mk D_{\mc X'}=
\bigcup_{\mbi l} (\Psi'^{\ast}
\mk D_{\mc X'^{\text{rig}}})\otimes \mo_{\mc X',\mbi l}$,
where $\mbi l$ runs over the set $\Z^r\cap \prod_{i=1}^r[0,a_i)$, and we define 
$\Psi'^{\ast}\mk D_{\mc X'^{\text{rig}}}$ in an obvious way.


\begin{lem}\label{reduction}
Let $\X$ be  a toric DM stack and 
$\mc X'^{\text{rig}}$ the rigidification of  
the root stack $\mc X'=\sqrt[\mbi a]{\mbi D/\mc X}$
as above. 
If $\mk D_{\mc X'^{\text{rig}}}$ generates 
$D^b(\mc X'^{\text{rig}})$, then $\mk D_{\mc X}$ 
generates $D^b(\mc X)$.
\end{lem}

\proof
Note that
by Lemmas~\ref{lem:rootc-vanish} and \ref{lem}, 
it is enough to show that $\mk D_{\mc X'}$ 
generates $D^b(\mc X')$.
The assumption and Theorem \ref{thm:orb-vs-stack} 
implies that the set $\mk D_{\mc X'}=
\bigcup_{\mbi l} (\Psi'^{\ast}
\mk D_{\mc X'^{\text{rig}}})\otimes \mo_{\mc X',\mbi l}$
generates $D^b(\X')$. 
Hence we obtain the assertion.
\endproof



\section{Root stacks of the projective plane}\label{sec:root-plane}

\subsection{Root stacks of weighted projective spaces}
First  let us recall the definition of 
weighted projective space $\mb P(\mbi a)$ for 
$\mbi a={}^t(a_1,\ldots , a_{n+1})\in\Z^{n+1}_{>0}$. Take a finitely generated abelian group $N:=\Z^{ n+1}/\Z\mbi a$
and  put $\beta$ to be the quotient map $\Z^{ n+1}\to N$.
Let us take the canonical basis $\mbi e_i$'s of $\Z^{ n+1}$ and 
consider the fan $\Delta$ in $N_{\mb R}$, consisting of cones $\sum_{i\neq j} \R_{\ge 0} \beta_\R(\mbi e_i)$ 
for $j=1,\ldots, n+1$ and their faces.
We often denote   by $\mb P(\mbi a)$ the toric DM stack associated  to the stacky fan 
$\Sigma=(\Delta,\beta)$. 
This is called a \emph{weighted projective space}.

%
%

For 
$
\mbi b={}^t(b_1,\ldots , b_{n+1})\in \Z_{>0},
$
we consider a root stack 
$\X:=\sqrt[\mbi b]{{\mbi D}/\mb P(\mbi a)}$.
Take a projective resolution of $N$
$$
0\to \Z\stackrel{\mbi a}\to \Z^{n+1}\to N=\Z^{ n+1}/\Z\mbi a\to 0.
$$
Then in terms in \S\ref{sec:2-2}, 
$\X$ is associated with matrices 
$$
A=\mbi a\in M(n+1,1),\quad B=\diag \begin{pmatrix}b_1,\ldots,b_{n+1}\end{pmatrix}\in M(n+1,n+1).
$$
%

Assume furthermore that 
$\mbi a=\mbi b=\mbi c={}^t(c,\ldots,c)\in\Z^{n+1}$ 
for an integer $c\in\Z$.
According to (\ref{pic}), we have 
$$
\Pic \sqrt[\mbi c]{\mbi D/\mb P(\mbi c)}\cong 
\frac{\Z\mbi f_{1}^{\ast}\oplus\cdots\oplus\Z\mbi f_{n+1}^{\ast}\oplus\Z\mbi g^{\ast}}
{\langle c\mbi f^{\ast}_1+c\mbi g^{\ast},\ldots,c\mbi f^{\ast}_{n+1}+
c\mbi g^{\ast}\rangle}\quad\mc O_{\X}(\sum l_i\mc D_i)_{k\Bar{\mbi g^{\ast}}}\mapsto
\sum l_i\Bar{\mbi f}^{\ast}_i+k\Bar{\mbi g^{\ast}},
$$
and hence we have an isomorphism 
\begin{equation}\label{deg}
h\colon\Pic \sqrt[\mbi c]{\mbi D/\mb P(\mbi c)}\to \Z\oplus\Z_{c}^{\oplus n+1}
\quad \sum l_i\Bar{\mbi f}^{\ast}_i+k\Bar{\mbi g^{\ast}}\mapsto
 \begin{pmatrix}\sum l_i-k\\\Bar{l}_1\\\vdots\\\Bar{l}_{n+1}\end{pmatrix}.
\end{equation}
For any $\mc L\in \Pic \sqrt[\mbi c]{\mbi D/\mb P(\mbi c)}$, we call the first component 
$\sum l_i-k$ of $h(\mc L)$ the \emph{degree} of $\mc L$ and denote it by $\deg \mc L$.


\subsection{Root stacks of the projective plane}\label{subsection:root-of-plane}
Consider  the root stack 
$\X=\sqrt[\mbi b]{\mbi D/\mb P^2}=
\X(\PP^2,b_1D_1+b_2D_2+b_3D_3)$, 
where $\mbi b={}^t(b_1.b_2,b_3)$ and $D_i=[z_i=0]$ in $\PP^2$.
The map 
$$
\beta\colon \Z^3
{\to} N=\Z^2
$$ 
is given by the matrix
$$
B=\begin{pmatrix}b_1&0&-b_3\\0&b_2&-b_3\end{pmatrix}.
$$
%

Moreover we put an additional assumption 
that  $b_1=b_2=b_3=c$. 
In this case by (\ref{formula}), we have 
$$
\mk D_{\X}=\left\lbrace \mo_{\mc X}\left(i(\mc D_1-\mc D_3)+j(\mc D_2-\mc D_3)+k\mc D_3 \right)\mid i,j\in[0,c),k=0,-1,-2\right\rbrace.
$$
By \cite[Proposition~5.1]{BHu}, 
this set forms a full strong exceptional collection on $\X$.
In particular we know that $\mk D_{\X}$ generates 
$
D^b(\X).
$
\begin{rem}
In general for root stacks $\mc X=\mc X(\PP^d,\sum_{i=1}^{d+1} b_iD_i)$ 
of the projective space of arbitrary dimension $d$ 
we see that $\mk D_{\mc X}$ generate $D^b(\mc X)$ as follows, 
where $b_i\ge 0$ and $D_i=[z_i=0]$.
By Lemma \ref{lem} we may assume that $b_1=\cdots=b_{d+1}$.
Then using \cite[Proposition~5.1]{BHu} the above argument holds 
for arbitrary dimension.  
\end{rem}


\section{Weighted blow ups and Frobenius morphisms}\label{sec:5}


\subsection{Weighted blow ups of toric DM orbifolds}\label{subsec:wbu}

Let us consider a toric DM orbifold $\mc X=\mc X_{\Delta}$
associated with a stacky fan $\Sigma=(\Delta,\beta)$.
Define primitive vectors $\mbi v_i$ in $N$  for $i=1,\ldots,s$ 
so that $\rho _i=\R_{\ge0}\mbi v_i$, namely, there exists $b_i\in \Z_{>0}$
 such that $b_i\mbi v_i=\beta (\mbi f_i)$.
Then as in \S \ref{sec:2-1}, $\X$ is also described as 
$
\X\bigl(X,\sum_{i=1}^s b_iD_i\bigr).
$
Take a cone $\sigma\in\Delta$ spanned by $\mbi v_1,\ldots ,\mbi v_l$, 
and  another primitive vector $\mbi v_{s+1}$ which is in the relative interior of $\sigma$;
there exist positive integers $h_i,m\in \Z_{>0}$ with coprime $h_i$'s satisfying
$$
m\mbi v_{s+1}=\sum_{i=1}^lh_i\mbi v_i.
$$
Then we have a new simplicial fan $\Delta'$ which is the subdivision of $\Delta$ 
obtained from the the star shaped decomposition of the fan $\Delta$,
by adding the ray  $\R_{\ge 0}\mbi v_{s+1}$.
Then we obtain a \emph{weighted blow up} of a toric variety $X=X_{\Delta}$:
$$
\Psi\colon X'=X_{\Delta'}\to X. 
$$

Suppose that positive integers $b_{s+1}$ and $c_i$ satisfy
\begin{equation}\label{eqn:wtb}
b_{s+1}\mbi v_{s+1}=\sum_{i=1}^lc_ib_i\mbi v_i.
\end{equation}
Then define  maps $\gamma_1,\gamma_2$ as
$$
\gamma_1\colon\bigoplus_{i=1}^{s+1}\Z \mbi f'_i\to\bigoplus_{i=1}^{s}\Z\mbi f_i\quad 
\mbi f'_i\mapsto
\begin{cases}
\mbi f_i&\text{ if }i\le s\\
\sum_{i=1}^lc_i\mbi f_i&\text{ if } i=s+1
\end{cases}
$$
and
$
\gamma_2=\id_N,
$
and take a toric DM orbifold $\mc X'=\X_{\Sigma '}$
associated with the stacky fan $\Sigma'=(\Delta', \beta\circ \gamma_1)$.
Note that $b_{s+1}\mbi v_{s+1}=\beta\circ \gamma_1(\mbi f'_{s+1})$.
Then we obtain a morphism
$$
\psi\colon \X'\to \X,
$$
called \emph{weighted blow-up} of a toric DM orbifold $\X$.


\begin{lem}\label{lem:van2*}
We have $\R\psi_{\ast}\mo_{\mc X'}=\mo_{\mc X}$.
\end{lem}

\begin{proof}
The assertion is shown in the proof of \cite[Theorem 9.1]{BH1}.
\end{proof}

Take a minimum integer $h>0$ such that each $h\frac{h_i}{b_i}$ is an integer.
We define $b_{s+1}:=hm$ and $c_i:=h\frac{h_i}{b_i}$. Then the equation (\ref{eqn:wtb})
is achieved, and hence
we obtain a morphism $\psi\colon \X'\to \X$ as above.
We call it the \emph{weighted blow-up of $\X$ associated with $\Psi$}.

\begin{rem}\label{rem:2dim}
In the two dimensional case, recall that a torus equivariant resolution
$$
\Theta\colon Y\to X
$$
 of singularities of  $X$ is a $t$-times composition of weighted blow-ups (cf. \cite[\S 2.6]{Fu}). 
Hence for such a resolution $\Theta$,
we can construct an associated morphism of toric DM orbifolds:
$$
\theta \colon \mc Y=\X(Y,\sum^{s+t} b_iD_i)\to \X=\X(X,\sum^s b_iD_i).
$$
\end{rem}


\subsection{Weighted blow ups of two dimensional toric DM stacks and Frobenius morphisms}
\label{subsec:blow-up}

We consider a two dimensional toric DM orbifold 
$\mc X=\mc X_{\Sigma}=\X(X,\sum_{i=1}^s b_iD_i)$
associated with a stacky fan $\Sigma=(\Delta,\beta)$ 
in a free abelian group 
$N=\Z\mbi e_1\oplus \Z\mbi e_2$.

Take a non-singular cone $\sigma\in \Delta$
generated by vectors $\mbi v_1$ and $\mbi v_2$, 
and consider the  new simplicial fan $\Delta'$ 
which is the subdivision of $\Delta$ 
obtained from the the star shaped decomposition of $\sigma$,
by adding the ray  $\R_{\ge 0}\mbi v_{s+1}$ in which we put 
$\mbi v_{s+1}:=\mbi v_1+\mbi v_2$.  
Consider the (weighted) blow-up 
$\Psi\colon X'\to X$ associated with the subdivision.
We may assume that $\mbi v_1=\mbi e_1$ and $\mbi v_2=\mbi e_2$,
and assume furthermore that $b_1=b_2$, which we denote by $c$. 
Let us define the blow-up
$\psi\colon \X'\to \X$ 
associated with $\Psi$.
Then we have 
\begin{equation}\label{picbup}
\psi^{\ast}\mc D_i=\mc D'_i+\mc D'_{s+1} \mbox{ for } i=1,2, \mbox{ and }
\psi^{\ast}\mc D_i=\mc D'_i \mbox{ for } i \ne 1,2. 
\end{equation}
Moreover  we have $b_{s+1}=c$.
By $\mbi v_{s+1}=\mbi e_1+\mbi e_2$, 
we have an isomorphism
$$
\mc{D}'_{s+1}\cong \sqrt[\mbi c]{\mbi D/\mathbb{P}(c,c)}
$$ 
for $\mbi c={}^t(c,c)$ by Lemma~\ref{stackisom}.
Let us define $\mc Q=[z_1=z_2=0]$.
Then we have the following diagram:
$$
\xymatrix{ 
{\X'} \ar[r]^{\psi }&\mc X \\
\mc D'_{s+1} \ar@{^{(}->}[u]^{j'} \ar[r]^p  &\mc Q  \ar@{^{(}->}[u]_{j}   .
}
$$
%


\begin{lem}\label{orlov}
There exists a semi-orthogonal decomposition
$$
D^b({\X'})=\langle j'_{\ast}\lbrace\mc L\in\Pic \mc D'_{s+1}\mid \deg\mc L=-1\rbrace, \mathbb{L}\psi^{\ast}D^b(\mc X)\rangle.
$$
\end{lem}

\proof
Let us denote the category in the r.h.s. by $\mc T$.
We show below that any objects in 
${}^{\perp}\mc T$ are isomorphic to $0$.
This completes the proof, since 
$\mc T$ is an admissible subcategory 
of $D^b(\mc X')$.

By Lemma~\ref{stackisom}, we have an isomorphism 
$$
\mathcal{ Q}\cong
\bigr[Q/\left(\Z_{c}\times\Z_{c}\right)\bigr],
$$
where $Q$ is a point.
By a similar argument in \cite[Theorem~4.3]{Or}, 
we see that for an object $A\in {}^{\perp}\mc T$,
there exists an object $B$ of $D^b(\mc Q)$ satisfying
\begin{equation}\label{adj1}
\mathbb{L}j'^{\ast}A={p}^{\ast}B,
\end{equation} 
and that $B\cong 0$ implies $A\cong 0$. 
We have
$$
\mo_{\mc Q,\mbi i}=j^{\ast}\mo_{\X}(i_1\mc D_1+i_2\mc D_2)
$$
for $\mbi i={}^t(i_1,i_2)\in \Z^{2}$. 
Since every object of $D^b(\mc Q)$ is a direct sum of
$\mc O_{\mc Q, \mbi i}[l]$ for some $l\in\Z$,
it is enough to show 
$\Hom_{\mc Q}(B,\mo_{\mc Q,\mbi i}[l])=0$ 
for any $\mbi i$ and $l$.
Take an invertible sheaf
$$
\mc F:=\mc O_{\mc D'_{s+1}}(i_1\mc D'_1+i_2\mc D'_2)_{i_1+i_2}
$$ 
on $\mc D'_{s+1}$,
and then we have 
\begin{equation}\label{adj2}
\mathbb{R}p_{\ast}\mc F\cong
\mc O_{\mc Q, \mbi i},\quad j'_{\ast}\mc F\cong
\mathbb{L}\psi^{\ast}j_{\ast}\mc O_{\mc Q, \mbi i}.
\end{equation}
%
The first isomorphism  
can be checked by similar computations 
in \cite[Proposition~4.1]{BHu}, and
the second is directly proved by the use of 
(\ref{emb}) and the Koszul resolution of 
$j_{\ast}\mc O_{\mc Q, \mbi i}$:
\begin{align*}
0\to\mc O_{\mc X}\left((i_1-1)\mc D_1+
(i_2-1)\mc D_2\right)\to 
\mc O_{\mc X}\left((i_1-1)\mc D_1+
i_2\mc D_2\right)\oplus
\mc O_{\mc X}\left(i_1\mc D_1+
(i_2-1)\mc D_2\right)\\ 
\to\mc O_{\mc X}\left(i_1\mc D_1+
i_2\mc D_2\right)\to
j_{\ast}\mc O_{\mc Q,\mbi i}\to 0.
\end{align*}
By $p^*\dashv \mathbb Rp_*, 
\mathbb Lj'^*\dashv j'_*$, 
(\ref{adj1}) and (\ref{adj2}), 
we get equalities
$$
\Hom_{\mc Q}(B,\mc O_{\mc Q, \mbi i}[l])=
\Hom_{\mc D_{s+1}}(\mathbb{L}j'^{\ast}A,\mc F[l])=
\Hom_{{\X'}}(A,\mathbb{L}\psi^{\ast}j_{\ast}
\mc O_{\mc Q, \mbi i}[l])=0,
$$
since we have 
$\mathbb{L}\psi^{\ast}j_{\ast}\mc O_{\mc Q, \mbi i}
\in \mc T$.
This implies $B\cong 0$, which completes the proof.
\endproof


\begin{lem}\label{keylemma}
If $\mk D_{\mc X}$ generates $D^b(\mc X)$, 
then $\mk D_{{\X'}}$ generates $D^b({\X'})$.
\end{lem}

\proof 
For $\mbi u\in \R^2$, we define 
$$
\mathcal{D}_{\mbi u}:=\sum_{i=1}^s 
\lfloor(\mbi u, b_i\mbi v_i)\rfloor 
\mathcal{D}_i,\quad
\mc D'_{\mbi u}:=\sum_{i=1}^{s+1} 
\lfloor(\mbi u, b_i\mbi v_i)\rfloor 
\mc D'_i.
$$
Recall that $ \mk D_{\mc X}=
\lbrace\mo_{\mc X}\left(\mc D_{\mbi u}\right)
\mid \mbi u\in [0,1)^2 \rbrace$ and $ \mk D_{\mc X'}=
\lbrace\mo_{\mc X'}\left(\mc D'_{\mbi u}\right)
\mid \mbi u\in [0,1)^2 \rbrace$
by (\ref{formula}).

Let $\mc T'$ denote ${j'_\ast}
\lbrace\mc L\in\Pic \mc D'_{s+1}\mid 
\deg\mc L=-1\rbrace$, 
and first we  show that 
$
\mc T' \subset \langle\mk D_{\X'}\rangle.
$
Note that 
$$
h(\mo_{\mc D'_{s+1}}
\left(l_1\mc D'_1+l_2\mc D'_2+k\mc D'_{s+1}\right))
=
\begin{pmatrix}
l_1+l_2-k\\
\Bar{l}_1\\
\Bar{l}_2\end{pmatrix}
\in \Z\oplus\Z_{c}\oplus\Z_{c}
$$ 
for the isomorphism $h$ in (\ref{deg}), hence 
$$
\mo_{\mc D'_{s+1}}\left(\mc D'_{\mbi u}\right)
\cong
\mo_{\mc D'_{s+1}}
\left(\lfloor cu_1\rfloor \mc D'_1+\lfloor cu_2\rfloor\mc D'_2+\lfloor cu_1+ cu_2\rfloor\mc D'_{s+1}\right)
\in \mc T'
$$ 
if 
\begin{equation}\label{eqn:c1c2}
\lfloor cu_1\rfloor+\lfloor cu_2\rfloor-\lfloor cu_1+ cu_2\rfloor=-1.
\end{equation}
%

%
Fix $l_1,l_2\in\Z$, and take a generic element $\mbi u$ in the set 
$$
\left\lbrace\mbi u={}^t(u_1,u_2)\in 
\R^2\ 
\Bigr|\  l_1< 
cu_1<l_1+1, l_2< cu_2<l_2+1, cu_1+cu_2=l_1+l_2+1
\right\rbrace.
$$
Then for  a sufficiently small real number $\e>0$,  we have
$$
\mo_{{\X'}}\left(\mc D'_{\mbi u-\epsilon{}^t(1,1)}\right)=\mo_{{\X'}}
\left(\mc D'_{\mbi u}-\mc D'_{s+1}\right).
$$
It follows from the exact sequence 
\begin{equation}\label{ex1}
0\to\mo_{{\X'}}\left(\mc D'_{\mbi u}-
\mc D'_{s+1}\right)\to\mo_{{\X'}}
\left(\mc D'_{\mbi u}\right)\to
\mo_{\mc D'_{s+1}}\left(\mc D'_{\mbi u}\right)\to 0
\end{equation}
that 
$\mo_{\mc D'_{s+1}}\left(\mc D'_{\mbi u}
\right)\in \langle \mk D_{{\X'}}\rangle$.
Since $\mbi u$ satisfies (\ref{eqn:c1c2}), 
we obtain $\mc T' \subset \langle\mk D_{\X'}\rangle$.

Next we show 
$
\mathbb{L}\psi^{\ast}\left(D^b(\mc X)\right)
\subset \langle  \mk D_{{\X'}}\rangle.
$
By the assumption $\langle  \mk D_{{\X}}\rangle=D^b(\X)$, 
it suffices to show that 
$\psi^{\ast}\mo_{\mc X}
(\mc D_{\mbi u})\in\langle \mk D_{{\X'}}\rangle$ 
for any $\mbi u\in \R^2$.
By (\ref{picbup}), we have 
$$
\psi^{\ast}\mo_{\mc X}(\mc D_{\mbi u})=
\begin{cases}
\mo_{{\X'}}(\mc D'_{\mbi u})              &\text{ if }\  \lfloor(\mbi u,c\mbi v_{s+1})\rfloor=\lfloor cu_1\rfloor + \lfloor cu_2\rfloor,\\ 
\mo_{{\X'}}(\mc D'_{\mbi u}-\mc D'_{s+1})&\text{ if }\ \lfloor(\mbi u,c\mbi v_{s+1})\rfloor=\lfloor cu_1\rfloor+\lfloor cu_2\rfloor+1.
\end{cases}
$$
Hence it is enough to consider the second case. 
Then we see by (\ref{eqn:c1c2}) that
$
\mo_{\mc D'_{s+1}}\left(\mc D'_{\mbi u}\right)
 \in \mc T' \subset \langle \mk D_{{\X'}}\rangle
$
and hence  
$\psi^{\ast}\mo_{\mc X}(\mc D_{\mbi u})\in 
\langle \mk D_{{\X'}}\rangle$ by (\ref{ex1}).

Now  Lemma~\ref{orlov} completes the proof.
\endproof


\section{Proof of Main theorem}\label{sec:mainthm}

The purpose of this note is to show the following.


\begin{thm}\label{thm}
For every two dimensional toric DM stack $\X$, 
the set $\mk D_\X$ generates $D^b(\X)$.
\end{thm} 

\proof
By Lemma~\ref{reduction}, we can reduce the proof of the statement 
to the case  $\X$ is an orbifold. 
Then there exist integers $b_i\in \Z_{>0}$ 
such that $\mc X=\mc X\left(X,\sum_{i=1}^sb_i D_i^X\right)$ 
as in \S \ref{sec:2-1}.

Consider a torus equivariant resolution of 
$X$ and the associated birational morphism 
between toric DM orbifolds as in Remark \ref{rem:2dim}.
Then by Lemmas \ref{lem} and \ref{lem:van2*}
we can reduce to the case $X$ is smooth.

We have a root construction morphism
$$
\mc X':=\mc X\left(X,c\sum_{i=1}^sD_i^X\right)
\to \mc X=\mc X\left(X,\sum_{i=1}^sb_i D_i^X\right)
$$ 
for $c=b_1b_2\cdots b_s$.
By lemmas \ref{lem} and \ref{lem:rootc-vanish}, 
it suffices to check the statement for $\mc X'$.

The strong factorization theorem 
(\cite[Theorem~1.28~(2)]{Od}) implies that  
there exists a smooth complete toric surface $Z$ and
morphisms $\Phi,\Psi$;
$$
X\stackrel{\Phi}\longleftarrow Z\stackrel{\Psi} \longrightarrow\PP^2,
$$
where $\Phi$ and $\Psi$ are compositions of 
torus equivariant blow-ups.
Associated with $\Phi$ and $\Psi$,
we obtain a compositions of blow-ups on toric DM orbifolds 
$$
\mc X'\stackrel{\phi}\longleftarrow \mc X \left (Z,c\sum_{i=1}^tD_i^Z\right)
\stackrel{\psi} \longrightarrow
 \mc X\left(\PP^2,cD_1^{\PP^2}+cD_2^{\PP^2}+cD_3^{\PP^2}\right).
$$ 
By Lemmas~\ref{lem}, \ref{lem:van2*} and \ref{keylemma},
it suffices to show the statement in the case where 
$\mc X$ is the root stack $\mc X
\left(\PP^2,cD_1^{\PP^2}+cD_2^{\PP^2}+cD_3^{\PP^2}\right)$
of the projective plane $\PP^2$,
which is already shown in  \S\ref{subsection:root-of-plane}.
\endproof

\begin{rem}
Suppose that $\mc X$ is a toric DM stack 
and consider the following Cartesian diagram:
\begin{equation*}
\xymatrix{ 
\mc X\times \PP ^n  \ar[r]^{p_1} \ar[d]_{p_2}&\PP^n\ar[d]^{q_2}\\
\mc X \ar[r]_{q_1}                       & \Spec \C    \\
}
\end{equation*}
Then we have 
$$
\R p_{2*}\mc O_{\mc X\times \PP ^n}=
\R p_{2*}p_1^*\mc O_{\PP ^n}=
q_1^*\R q_{2*}\mc O_{\PP ^n}=
\mc O_{\X}
$$
by the flat base change theorem. 
Hence by Lemma~\ref{lem},
if $\mk D_{\mc X\times \PP ^n}$ 
generates $D^b(\mc X\times \PP ^n)$,
then  $\mk D_{\mc X}$ generates $D^b(\mc X)$. 
In particular, $\mk D_{\mc X}$ generates 
$D^b(\mc X)$ for $0$ or $1$-dimensional 
toric DM stacks $\mc X$ by Theorem \ref{thm}.
\end{rem}

\begin{rem}
An obstruction to extend Main Theorem to higher dimension is the following. For higher dimensional cases, we can get a statement similar to Lemma 6.4 again, when the center of a blow-up is zero-dimensional. 
However, when the dimension of the center is greater than zero, we do not know whether a similar statement holds.
\end{rem}


\section{Full strong exceptional collections}
\label{sec:6}
In this section, 
we choose a full strong exceptional 
collection from the set $\mk D_{\X}$,  
in several examples of one or two dimensional toric stacks $\X$. 

\subsection{Examples}
First let us show Lemma \ref{lem:nef}. Let $\X$ be a toric DM orbifold and assume that
a coarse moduli space $X$ of $\X$ is smooth. Recall that we have a root construction morphism
$\pi\colon \X\to X$ as in \S \ref{sec:2-5}.
Then for any divisor $\mc D$ on $\X$, there exists a divisor $D$ on $X$ such that 
$\pi^* \mc O_X(D)=\mc O_\X(\mc D)^{\otimes m}$ for some $m>0$. Then we know that 
$\mc D$ is nef on $\X$ if and only if so is $D$ on $X$. Under this notation, we have the following:     

\begin{lem}\label{lem:nef}
If $\mc D$ is nef, then we have $H^i(\mc X,\mc O_{\mc X}(\mc D))=0$ for $i>0$.
\end{lem}

\proof
Since $\mc O_{\X}$ is a direct summand of $F_{m*}\mc O_{\X}$, we have 
\begin{align*}
H^i(\mc X,\mc O_{\mc X}(\mc D))\subset &H^i(\mc X,\mc O_{\mc X}(\mc D)\otimes F_{m*}\mc O_{\X})
\cong  H^i(\mc X,F^*_m\mc O_{\mc X}(\mc D))\\
\cong  &H^i(\mc X,\mc O_{\mc X}(\mc D)^{\otimes m}) \cong H^i(\mc X,\pi^* \mc O_X(D))=0,
\end{align*}
which completes the proof.
Note that the last equality is a consequence of  the nef vanishing  on toric varieties. 
\endproof

Put 
$$
\mk D^{\text{nef}}_{\X}:=\left\lbrace
\mathcal{O}_{\X}(\mc D)\in\mk D_{\X}\mid 
-\mc D \text{ is nef}\right\rbrace.
$$
If we have
$$
\mk D_{\X}\subset\langle
\mk D^{\text{nef}}_{\X}\rangle,
$$ 
then the set $\mk D^{\text{nef}}_{\X}$ 
forms a full strong exceptional collection 
by \cite[Lemma~3.8(i)]{Ue}, Lemma \ref{lem:nef} and Theorem~\ref{thm}.

Below we give calculation only in some typical 
cases and omit it in the rest.
We leave it to readers.  
 
(1)
For $\X=\PP(a_1,a_2,a_3)$ with $(a_1,a_2,a_3)=1$ or $\X=\X(\PP^1\times\PP^1, \sum b_iD_i)$ 
with arbitrary $b_1,b_2,b_3,b_4\in\Z_{>0}$, we have $\mk D^{\text{nef}}_{\X}=\mk D_{\X}$.
Hence $\mk D_{\X}$ forms a full strong exceptional collection. 

(2)  Blow-up at a point on $\PP^2$, and then 
we obtain the Hirzebruch surface $\mathbb F_1$.
For the toric DM orbifold $\X=\X(\mathbb F_1, D_1+D_2+2D_3+D_4)$, 
we have $\mk D^{\text{nef}}_{\X}\neq \mk D_{\X}$,
 where $D_1,D_3$ are fiber of the ruling and 
 $D_2, D_4$ are negative and positive sections respectively.
In fact we take vectors 
$$
\mbi v_1=\begin{pmatrix}1\\0\end{pmatrix}, \mbi v_2=\begin{pmatrix}0\\1\end{pmatrix}, \mbi v_3=\begin{pmatrix}-2\\2\end{pmatrix} \text{ and }\mbi v_4=\begin{pmatrix}0\\-1\end{pmatrix}\in N,
$$ 
corresponding to divisors $\mathcal{D}_1,
\mathcal{D}_2, \mathcal{D}_3$ and $\mathcal{D}_4$ 
on $\X$. 
Then we obtain 
$$
\mk D_{\X}=\{\mo_{\X},\mo_{\X}(-\mathcal{D}_3),\mo_{\X}(-2\mathcal{D}_3),\mo_{\X}(\mathcal{D}_3-\mathcal{D}_4),\mo_{\X}(-\mathcal{D}_4),\mo_{\X}(-\mathcal{D}_3-\mathcal{D}_4),\mo_{\X}(-2\mathcal{D}_3-\mathcal{D}_4)\}
$$
and $\mk D^{\text{nef}}_{\X}=
\mk D_{\X}\setminus 
\lbrace\mo_{\X}(\mathcal{D}_3-\mathcal{D}_4)\rbrace$.
In this case we can see that 
$\mo_{\X}(\mathcal{D}_3-\mathcal{D}_4)\in 
\langle\mk D^{\text{nef}}_{\X}\rangle$, and hence
$\mk D_{\X}\subset\langle\mk D^{\text{nef}}_{\X}\rangle$.

(3)
Finally we take the toric DM orbifold $\X$ defined by vectors 
$$
\mbi v_1=\begin{pmatrix}1\\0\end{pmatrix}, \mbi v_2=\begin{pmatrix}0\\1\end{pmatrix}, \mbi v_3=\begin{pmatrix}-2\\2\end{pmatrix}, \mbi v_4=\begin{pmatrix}-1\\0\end{pmatrix}\text{ and }\mbi v_5=\begin{pmatrix}0\\-1\end{pmatrix}
$$ 
in $N$ corresponding to divisors 
$\mathcal{D}_1,\mathcal{D}_2, \mathcal{D}_3,\mathcal{D}_4$ 
and $\mathcal{D}_5$ on $\X$. 
Then we obtain 
\begin{equation*}
\begin{split}
\mk D_{\X}=
&\{\mo_{\X},\mo_{\X}(-\mathcal{D}_3-\mathcal{D}_4),\mo_{\X}(-2\mathcal{D}_3-\mathcal{D}_4),\mo_{\X}
(\mathcal{D}_3-\mathcal{D}_5),\mo_{\X}(-\mathcal{D}_5),\\
&\mo_{\X}(-2\mathcal{D}_3-\mathcal{D}_4-\mathcal{D}_5),\mo_{\X}(-\mathcal{D}_3-\mathcal{D}_4-\mathcal{D}_5),\mo_{\X}
(-\mathcal{D}_4-\mathcal{D}_5),\mo_{\X}(\mathcal{D}_3-\mathcal{D}_4-\mathcal{D}_5)\}
\end{split}
\end{equation*}
and $\mk D^{\text{nef}}_{\X}=\mk D_{\X}
\setminus \{ \mc L_1,\mc L_2,\mc L_3\}$,
where we define
$$
\mc L_1:=\mc O_{\X}(\mathcal{D}_3-\mathcal{D}_5), 
\mc L_2:=\mc O_{\X}(-\mathcal{D}_3-\mathcal{D}_4)\text{ and }
\mc L_3:=\mc O_{\X}(\mathcal{D}_3-\mathcal{D}_4-\mathcal{D}_5).
$$
In this case, since $\#\mk D^{\text{nef}}_{\X}=
6<\rk K(\X)=7$, the set $\mk D^{\text{nef}}_{\X}$ 
does not form a full exceptional collection 
(see Remark \ref{rem:Grothendieck}),
and thus we know $\mk D_{\X}\not\subset\langle\mk D^{\text{nef}}_{\X}\rangle$.
However we can see the subset 
$$
\mc S:=\mk D_{\X}\setminus \lbrace\mo_{\X}(-2\mathcal{D}_3-
\mathcal{D}_4),\mo_{\X}(-2\mathcal{D}_3-\mathcal{D}_4-
\mathcal{D}_5)\rbrace
$$
of $\mk D_{\X}$ is a full strong exceptional 
collection as follows.

By the exact sequence
$$
0\to\mo_{\X}\to\mo_{\X}(\mathcal{D}_3-\mathcal{D}_5)\to
\mo_{\mathcal{D}_3}(\mathcal{D}_3)\to 0,
$$
we have $\mo_{\mathcal{D}_3}(-\mathcal{D}_3-\mathcal{D}_4)=\mo_{\mathcal{D}_3}(\mathcal{D}_3)\in\langle \mc S\rangle$.
Moreover by the exact sequence
\begin{equation}\label{ex2}
0\to\mo_{\X}(-2\mathcal{D}_3-\mathcal{D}_4-i\mc D_5)\to\mo_{\X}(-\mathcal{D}_3-\mathcal{D}_4-i\mc D_5)
\to\mo_{\mathcal{D}_3}(-\mathcal{D}_3-\mathcal{D}_4)\to 0
\end{equation}
for $i=0,1$,
we have $\mo_{\X}(-2\mathcal{D}_3-\mathcal{D}_4
-i\mc D_5)\in\langle \mc S\rangle$.
Hence we obtain  $ \mk D_{\X} \subset 
\langle \mc S \rangle$ and hence  $\mc S$ 
generates $D^b(\X)$.
According to  \cite[Proposition~4.1]{BHu}, 
we show below that  for any $\mc L,\mc L'\in \mc S$,
$\Ext^i_{\X}\left(\mc L,\mc L' \right)= 0$ 
for $i\ne 0$ to conclude that $\mc S$ is 
a full strong exceptional collection.

For every $\mbi r=(r_i)\in\Z^5$, 
we denote by $\Supp(\mbi r)$ 
the simplicial complex on five vertices 
$\{1,\ldots,5\}$ which consists of all subsets 
$J\subset\{1,\ldots,5\}$ such that $r_i\ge 0$ for all $i\in J$
and there exists a cone in the fan $\Delta$ 
determining $\X$ that contains all $\mbi v_i, i\in J$.
By \cite[Proposition~4.1]{BHu}
\footnote{There is a typo in the statement 
\cite[Proposition~4.1]{BHu}.
Precisely, the cohomology $H^p(\mc X,\mc L)$
is obtained by $(\rk N-p-1)$-th reduced homology of $\Supp(\mbi r)$.} 
we see that for any line bundle $\mathcal{L}$ on $\X$, we have 
\begin{equation}\label{eqn:BHu}
\dim H^1(\X,\mathcal{L})=\sum_{\mbi r}
\left(\#\lbrace 
\text{ connected components of } \Supp(\mbi r) 
\rbrace-1\right),
\end{equation}
where the summation is taken over the set of 
all $\mbi r=(r_i)\in\Z^5$ such that 
$\mc O_{\X }(\sum_{i=1}^5r_iD_i)\cong\mathcal{L}$.

For example, we know from (\ref{eqn:BHu}) that 
\begin{align*}
\Ext^1_{\X}\left(\mc L_1,\mo_{\X}(-2\mathcal{D}_3-\mathcal{D}_4)\right)
=H^1(\X,\mo_{\X}(-3\mc D_3-\mc D_4+D_5))\neq 0,
\end{align*}
since 
$-3\mc D_3-\mc D_4+\mc D_5\sim  -\mc D_1-\mc D_3+\mc D_5$
and $\Supp({} ^t (-1,0,-1,0,1))$ has $2$ connected components.
By \cite[Proposition~4.1]{BHu} and easy but 
tedious computations as above, 
we can see that $\mc S$ forms a full strong
exceptional collection.

Moreover again by (\ref{eqn:BHu}) we have
\begin{equation*}
\begin{split}
\Ext^1_{\X}\left(\mc L_i,\mo_{\X}(-2\mathcal{D}_3-\mathcal{D}_4)\right)\neq 0 \text{ and }
\Ext^1_{\X}\left(\mc L_i,\mo_{\X}(-2\mathcal{D}_3-\mathcal{D}_4-\mathcal{D}_5)\right)\neq 0
\end{split}
\end{equation*}
for all $i=1,2,3$.  
Hence we conclude that $\mc S$ is a unique subset 
of $\mk D_{\X}$ which forms a full exceptional collection.

We remark that any two dimensional toric Fano orbifold 
has a full strong exceptional collection of 
line bundles by \cite[Theorem~7.3]{BHu}.


\subsection{Toric DM orbifolds with $\rank \Pic \X=1.$}

As an application of Main Theorem, we also have the following:


\begin{thm}
Let $\X$ be an one or two dimensional toric DM stack.
Suppose furthermore that its rigidification $\mc X^{\rig}$ 
has the  Picard group of rank $1$.
(Notice that this condition is automatically satisfied when $\X$ is $1$-dimensional.)
Then the set $\mk D_\X$ forms a full strong exceptional collection. 
\end{thm}

\begin{proof}
By Corollary \ref{cor:orb-vs-stack}, it suffices to show the statement for toric DM orbifolds.
Take the linear function $\deg\colon \Pic \X\to \Z$ which takes value $1$ on 
the positive generator of $\Pic \X$. 

We first show that $\mk D_\X$ is contained in the set $\mk L$ 
of line bundles $\mc L$ satisfying $\deg K_{\X}<\deg \mc L\le 0$.  
Take an arbitrary line bundle $\mc L$ in $\mk O_\X$. 
Then by (\ref{formula}), $\mc L$ is of the form 
$
\mo_\X\left( \sum_i \lfloor  \frac{(\mbi u,b_i\mbi v_i)}{m} \rfloor  \mc D_i \right)
$ 
for some $m\in \Z_{>0}$ and $\mbi u\in \Z^n \cap [0,m-1]^n$. Since 
$$
(\mbi u,b_i\mbi v_i)=m \lfloor \frac{(\mbi u,b_i\mbi v_i)}{m}\rfloor +r_i
$$
for some integers $r_i$ with $0\le r_i < m$, the divisor $m \sum_i \lfloor  \frac{(\mbi u,b_i\mbi v_i)}{m} \rfloor  \mc D_i$ is linearly equivalent
to the divisor $-\sum_i r_i  \mc D_i$. 
In particular, we have $\mc L^{\otimes m}\cong \mc O_\X( -\sum_i r_i  \mc D_i)$, and hence
we conclude that $\deg K_\X< \deg \mc L\le 0$.

By \cite[Proposition 5.1]{BHu} the set $\mk L$ 
forms a full strong exceptional collection. 
Since $\mk D_\X$ generates $D^b(\X)$, 
the set $\mk D_\X$ coincides with the collection they chose.
Therefore the set $\mk D_\X$ forms a full strong exceptional collection.
\end{proof}

\noindent
Ryo Ohkawa

Research Institute for Mathematical Sciences, 
Kyoto University, 
Kyoto, 606-8502, Japan

{\em e-mail address}\ : \  ohkawa@kurims.kyoto-u.ac.jp

\ \\

\noindent
Hokuto Uehara

Department of Mathematics
and Information Sciences,
Tokyo Metropolitan University,
1-1 Minamiohsawa,
Hachioji-shi,
Tokyo,
192-0397,
Japan 

{\em e-mail address}\ : \  hokuto@tmu.ac.jp

\end{document}